\documentclass[11pt]{amsart}
\usepackage{graphicx}
\usepackage[USenglish,english]{babel}
\usepackage[T1]{fontenc}
\usepackage[latin1]{inputenc}
\usepackage{amsfonts}
\usepackage{amssymb}
\usepackage{amsthm}
\usepackage{amsmath}
\usepackage{fancyhdr}

\usepackage{float}
\usepackage{enumerate}
\usepackage{accents}
\usepackage{mathtools}
\usepackage{mathrsfs}
\usepackage{comment}
\usepackage{afterpage}
\usepackage{bbm}
\usepackage{thm-restate}
\usepackage{dsfont}
\usepackage{stmaryrd}
\usepackage{mleftright}
\usepackage{subfig}
\usepackage{faktor}
\usepackage{graphicx}
\usepackage{marvosym}
\usepackage[all]{xy}
\usepackage[colorlinks=true,allcolors=blue!80!black]{hyperref}
\usepackage{xcolor}

\newtheorem{theorem}{Theorem}
\numberwithin{theorem}{section}

\newtheorem{lemma}[theorem]{Lemma}

\newtheorem*{claim*}{Claim}
\newtheorem{proposition}[theorem]{Proposition}

\newtheorem*{question*}{Question}
\newtheorem*{theorem*}{Theorem}

\theoremstyle{remark}
\newtheorem{remark}[theorem]{Remark}

\newtheorem*{remark*}{Remark}

\theoremstyle{definition}
\newtheorem{definition}[theorem]{Definition}

\newtheorem{notation}[theorem]{Notation}

\newtheorem*{warning*}{Warning}
\newtheorem*{convention*}{Convention}
\newtheorem*{example*}{Example}

\newcommand{\Id}{\text{Id}}
\newcommand{\Aut}{\text{Aut}}

\newcommand{\id}{\text{id}}

\newcommand{\reduc}{\mathrm{red}}

\newcommand{\pre}[2]{\prescript{}{#1}{ #2}}

\DeclareMathOperator{\Stab}{Stab}
\DeclareMathOperator{\CAT}{CAT(0)}

\newcounter{mcomments}

\title{The fundamentals of cubical isometry groups}

\author{Merlin Incerti--Medici}
\address{Universit\"at Wien, Austria}
\email{merlin.medici@gmail.com}

\begin{document}
\maketitle

\begin{abstract}
We develop the fundamental theory to study cubical isometry groups as totally disconnected, locally compact groups. We show how cubical isometries are determined by their local actions and how this can be applied in explicit constructions. These results are closely related to some of the authors recent work on cubical isometries. We reformulate and generalize these previous results in a way that is necessary and more suited for upcoming applications.



\end{abstract}

\setcounter{tocdepth}{1}
\tableofcontents

\section{Introduction} \label{sec:introduction}

\subsection{History and motivation}

Let $S$ be a compact, special, non-positively curved cube complex and denote its universal covering by $X$. 
Let $Y$ be the universal covering of a compact, special, non-positively curved cube complex as well. We will refer to spaces like $X$ and $Y$ as $\CAT$ cube complexes with special, compact quotient. 
In this article, we develop the basic techniques and results used to study cubical isometries $X \rightarrow Y$. To do so, we also cover the basic theory of isomorphisms between rooted, self-similar trees. Putting $X = Y$, these techniques can be applied to the group $\Aut(X)$ of automorphisms on $X$, that is, the group of cubical isometries on $X$. We equip $\Aut(X)$ with the compact-open topology that stems from the $\CAT$-metric on $X$.

There are many different classes of locally finite combinatorial geometric objects whose automorphism groups have been a subject of interest over the last few decades. These classes include regular trees \cite{Tits70}, negatively curved polyhedral complexes \cite{HaglundPaulin98}, right-angled buildings \cite{Caprace14, MedtsSilvaStruyve18, CapracedeMedts24}, self-similar trees \cite{HartnickMedici23a}, and $\CAT$ cube complexes with special, compact quotient \cite{HartnickMedici23b}. These classes all share that their automorphism groups, when equipped with the compact-open topology, are totally disconnected, locally compact (tdlc) groups. Furthermore, if the combinatorial object is sufficiently symmetric, then its automorphism group is frequently non-discrete and can be described rather explicitly. This is of interest for several reasons:

Let $\Lambda$ denote the fundamental group of $S$. Our assumptions imply that $\Lambda$ is special in the sense of Haglund and Wise {\cite[Lemma 2.6]{HaglundWise07}} (see Section \ref{sec:cubical.labeling.equivalent.to.special}). Special groups have been of significant importance in geometric group theory and low-dimensional topology in recent decades because, among other reasons, they have been involved in the solution of conjectures concerning $3$-manifolds and one-relator groups \cite{Agol13, BaumslagFineRosenberger19, Wise21}. By definition, $X$ admits a free, cocompact, cubical action by $\Lambda$ and thus $\Lambda$ embeds as a discrete and cocompact subgroup into $\Aut(X)$. In other words, for any special group, $\Aut(X)$ is a lattice envelope, that is, a locally compact group that contains $\Lambda$ as a lattice. The concept of a lattice envelope was introduced by Furstenberg \cite{Furstenberg67} when he was studying to what extent a lattice envelope is determined by the lattice. A central question in the study of lattice envelopes is the following: Given a group $\Lambda$, can we describe all its lattice envelopes up to an appropriate equivalence? In particular, when does a group $\Lambda$ have a unique lattice envelope? A famous example of unique lattice envelopes is known as Mostow's Strong Rigidity (see \cite{Mostow73}). For a detailed discussion on lattice envelopes and the problem of finding all lattice envelopes for a given group $\Lambda$, see \cite{BaderFurmanSauer20}. With these questions in mind, the study of $\Aut(X)$ becomes interesting, as it provides a natural way to embed a special group $\Lambda$ into a, frequently non-discrete, tdlc group, providing a natural lattice envelope whose relationship to $\Lambda$ we can investigate.

A second motivation is that non-discrete tdlc groups form one of three important classes of locally compact groups: Lie groups, discrete groups, and non-discrete tdlc groups. The theory of locally compact groups can be broken down into the study of these three classes (see \cite{Castellano23}) which have all been studied extensively with varying results. When it comes to non-discrete tdlc groups, a lot of focus was put on the study of certain subclasses like linear groups over local fields and automorphism groups of certain buildings. Regarding general theory, after a period of stagnation, recent years have brought some interesting developments, like the introduction of scale-functions and the structure lattice. These developments suggest that a general structure theory about non-discrete tdlc groups may be attainable \cite{CapraceReidWillis17a, CapraceReidWillis17b}. The automorphism groups of $\CAT$ cube complexes with special, compact quotient are, as we will see, a class of examples that can be described very explicitly. Apart from being interesting in their own right, this makes them excellent examples to consider under the lens of this new general theory of non-discrete tdlc groups.\\

In \cite{HartnickMedici23a, HartnickMedici23b}, Hartnick and the author developed a theory for the study of automorphism groups of $\CAT$ cube complexes with special, compact quotient. Our focus in those works was on normal cube paths and the analogy to self-similar groups. When the author of this paper continued his work on cubical isometries with an outlook towards tdlc groups, it turned out that the setup of the theory was not ideal and some of the established results were not sufficiently general to deal with key questions the author was working on. This paper exists to remedy this. Its objective is to compile the fundamental technical results in a systematic framework that will be useful for several upcoming works of the author. Furthermore, the author uses the opportunity to create a self-contained write-up of the basic results that can serve as a point of reference for anyone intending to study cubical isometry groups.

While there are many small changes to the setup and formulation of the theory compared to the previous work, there are two changes and one generalisation that are at the core of this paper. We briefly describe these three points and their consequences. First, the previous papers were reliant on the notion of normal cube paths, a type of paths in $\CAT$ cube complexes. In this paper, we develop the theory using edge-paths instead of normal cube paths. This cuts the length and complexity of every geometric argument in half, as we no longer have to deal with diagonals. It also allows us to rely more on established results about cube complexes, as edge-paths are far more commonly studied than paths that include diagonals. In addition, it simplifies the notation for two important trees that we will introduce below.

Second, the definition of cubical portraits (Definition \ref{def:cubical.portrait.on.subtree}) is less convoluted than its counterpart in \cite{HartnickMedici23b}. This change also simplifies the formulation of Theorems \ref{thm:characterization.of.cubical.portraits} and \ref{thm:canonical.portrait.extension} and allows us to restructure the proofs of these two theorems to extract useful substeps. In particular, we obtain Lemma \ref{lem:moves.are.carried.over}, which is useful for upcoming work of the author concerning the extension of partially defined isometries. We also highlight that the definition of cubical portraits has been generalised to subtrees, a situation that has not been considered at all in \cite{HartnickMedici23b} and that is again highly relevant for upcoming work concerning the extension of partially defined isometries.

Finally, 
we generalise the results obtained in \cite{HartnickMedici23a, HartnickMedici23b} in the following way: Several key technical results of the previous papers concern cubical isometries from $X$ to itself that have a fixed point. In this paper, we generalise most of these results to cubical isometries $X \rightarrow Y$ between cube complexes. In particular, the requirement of having a fixed point is removed. Similarly, we generalise the previous theory of rooted tree-automorphism on self-similar trees to a theory of rooted tree-isomorphisms. We do this because of a third motivation that stems from universal groups in the sense of Burger and Mozes. Intuitively, given a subgroup $F$ of the permutation group of $N$ elements, the universal group of $F$ in the sense of Burger and Mozes is the group of automorphisms on the $N$-regular tree whose local action at every vertex has to lie in $F$ {\cite[Section 3]{BurgerMozes00}}. As we will see, the approach to cubical isometries presented in this paper points very clearly at a notion of universal groups that has been adapted to the context of cubical isometries. However, such a notion requires describing cubical isometries in terms of their local actions regardless of whether they have a fixed point or not (since elements of universal groups generally do not have a fixed point). To prepare for future work in this direction, we generalise the theory to cubical isometries that may not have fixed points in this paper.\\ 


The approach to $\Aut(X)$ that we present in this article relies on encoding cubical isometries as isomorphisms between rooted trees. More specifically, let $g \in \Aut(X)$ and choose a vertex $o \in X$. We may now consider the space of edge-paths in $X$ that start at $o$. These edge-paths naturally carry the structure of a tree that we denote by $T_o^{edge}$. The automorphism $g$ thus induces an isomorphism of rooted trees $T_o^{edge} \rightarrow T_{g(o)}^{edge}$. Since cubical isometries are uniquely determined by their action on the $1$-skeleton, this isomorphism of rooted trees fully determines the cubical isometry. We will encode elements of $\Aut(X)$ in this way. To do so, we need to establish some foundations of rooted tree isomorphisms, which we briefly discuss in the next subsection.

\subsection{Rooted tree isomorphisms of self-similar trees}

Let $\Gamma$ be a finite, oriented graph and $o$ a marked vertex in $\Gamma$. The set of oriented edge-paths in $\Gamma$ that start at $o$ carries the structure of a rooted tree $T_{\Gamma, o}$, where the root is the path that stays at $o$. A tree is called {\it self-similar} if it is the path-space of such a pair $(\Gamma, o)$. Self-similar trees have appeared in several subjects, ranging from probability theory over spectral graph theory to geometric group theory. They have also appeared under several names, notably as periodic trees \cite{Lyons90}, as trees with finitely many cone types \cite{NagnibedaWoess02, KellerLenzWarzel12b, KellerLenzWarzel14}, and as self-similar trees \cite{BelkBleakMatucci21, HartnickMedici23a}. We primarily need a theory of isomorphisms between self-similar trees, whose basics have been developed in \cite{HartnickMedici23a}.

Let $\Sigma$ be a finite set and let $\Sigma^*$ denote the set of all word over the alphabet $\Sigma$. Given a labeling of the edges in $\Gamma$ by elements of $\Sigma$, we can lift this labeling to an edge-labeling on the tree $T_{\Gamma,o}$. This edge-labeling allows us to identify every path $\gamma$ in $\Gamma$ that starts at $o$ with the word $v \in \Sigma^*$ that is spelled out by the labels along the trajectory of $\gamma$. In particular, we can identify vertices of $T_{\Gamma,o}$ with certain words over $\Sigma$.

The edge-labeling induced on $T_{\Gamma,o}$ has the property that for any vertex $v$ in $T_{\Gamma,o}$, now seen as a word over $\Sigma$, the path from the root to $v$ spells out the word $v$ (see Section \ref{sec:trees} for details). We may thus think of a word $v$ both as a vertex in $T_{\Gamma,o}$ and as the shortest edge-path in $T_{\Gamma,o}$ from the root to $v$. The path that stays at $o$ is represented by the empty word, denoted $\epsilon$. We thus denote the root of $T_{\Gamma, o}$ by $\epsilon$.

\begin{definition}
	Let $T$ and $T'$ be two self-similar trees with edge-labelings induced as described above and let $\Sigma$ and $\Sigma'$ denote the respective set of labels. A {\it portrait on $T$ to $T'$} is a family $(\pre{v}{\sigma})_{v \in V(T)}$ of injective maps $\pre{v}{\sigma} : \Sigma_v \hookrightarrow \Sigma'$, (spoken ``$\sigma$ at $v$'') where $V(T)$ denotes the set of vertices of $T$.
\end{definition}

Given a portrait on $T$ to $T'$, we define the map
\[ \sigma: V(T) \rightarrow \Sigma'^* \]
\begin{equation} \label{eq:intro.equation}
	s_1 \dots s_n \mapsto \pre{\epsilon}{\sigma}(s_1) \dots \pre{s_1 \dots s_{i-1}}{\sigma}(s_i) \dots \pre{s_1 \dots s_{n-1}}{\sigma}(s_n).
\end{equation}

The map $\sigma$ induces an injective morphism of rooted trees if and only if its image is contained in the set $V(T')$. Theorem \ref{thm:characterising.tree.isomorphisms} provides a characterisation of such portraits in terms of a semi-local condition at every vertex $v \in V(T)$. 
%
%
%
This characterisation allows us to construct rooted tree isomorphisms between self-similar trees by constructing portraits that satisfy this semi-local condition. It serves as a basis to study automorphism groups of self-similar trees and generalises ideas used in the study of automorphisms of regular rooted trees. It is an interesting open question to what extent the theory of automorphisms of regular rooted trees -- especially the theory of self-similar groups -- can be expanded to self-similar trees.

\subsection{Cubical isometries}

We return to a $\CAT$ cube complex $X$ with a special, free, and cocompact action $\Lambda \curvearrowright X$. Our first step is to relate cubical isometries on $X$ to rooted tree isomorphisms between self-similar trees. For any vertex $o$ in $X$, consider the set of all edge-paths in $X$ that start at $o$. This set carries the structure of a tree that we denote by $T_o^{edge}$. As we will see in Lemma \ref{lem:regular.language}, this tree is self-similar. Since any cubical isometry $g \in \Aut(X)$ is determined by its action on the vertices, it is certainly determined by the rooted tree isomorphism $T_o^{edge} \rightarrow T_{g(o)}^{edge}$ that it induces. Note that the vertex $o$ can be chosen arbitrarily here. If $g$ has a fixed point, it is generally useful to pick $o$ to be such a fixed point as $g$ then induces an automorphism $T_o^{edge} \rightarrow T_o^{edge}$.

An oriented edge in $X$ is an edge together with a chosen orientation. We denote the set of oriented edges in $X$ by $\overrightarrow{\mathcal{E}}(X)$ and the set of oriented edges starting at a vertex $o$ by $\overrightarrow{\mathcal{E}}(X)_o$. Given an oriented edge $e$, we write $e^{-1}$ to denote the edge obtained by reversing orientation. Two oriented edges are called parallel, if they cross the same hyperplane ``in the same direction'' (see Section \ref{subsec:edge.paths}). We now label the oriented edges of a cube complex in a way that is sensitive to the geometry of the cube complex. This labeling induces a labeling on the tree $T_o^{edge}$ with which we will describe our rooted tree isomorphisms. The following definition describes the type of edge-labeling that we will work with.

\begin{definition}
	A {\it cubical edge-labeling} on $X$ is a surjective map $\ell : \overrightarrow{\mathcal{E}}(X) \rightarrow \Sigma$ satisfying the following properties.
	\begin{enumerate}
		\item For every vertex $o \in X^{(0)}$, the restriction $\ell \vert_{\overrightarrow{\mathcal{E}}(X)_o}$ is injective.
		
		\item If $e, e' \in \overrightarrow{\mathcal{E}}(X)$ are parallel oriented edges, then $\ell(e) = \ell(e')$.
				
		\item Let $o, o' \in X^{(0)}$, let $e, f \in \overrightarrow{\mathcal{E}}(X)_o$, let $e', f' \in \overrightarrow{\mathcal{E}}(X)_{o'}$ such that $\ell(e) = \ell(e')$ and $\ell(f) = \ell(f')$. Then $e$ and $f$ span a square if and only if $e'$ and $f'$ do.
		
		\item There exists a fixed-point free involution $\cdot^{-1} : \Sigma \rightarrow \Sigma$ such that for every $e \in \overrightarrow{\mathcal{E}}(X)$, $\ell(e^{-1}) = \ell(e)^{-1}$.

	\end{enumerate}
	
	We call a cubical edge-labeling {\it $\Lambda$-invariant} if for all $g \in \Lambda$ and all $e \in \overrightarrow{\mathcal{E}}(X)$, we have $\ell(e) = \ell(g(e))$.
\end{definition}

Given a cubical edge-labeling and a vertex $v \in T_o^{edge}$, we write $\Sigma_v^{edge}$ for the set of labels appearing on the outwards pointing edges at $v$ in the tree $T_o^{edge}$. The main result about the existence of cubical edge-labelings is Theorem \ref{thm:characterising.existence.of.cubical.edge.labelings} which states that $\Lambda$-invariant cubical edge-labelings exist if and only if $X$ is the universal covering of a compact, non-positively curved, special cube complex with fundamental group $\Lambda$.


Cubical edge-labelings have some history. They appeared under the name {\it special colorings} in \cite{Genevois21} where their relationship to special cube complexes is discussed(cf.\,{\cite[Theorem 4.1]{Genevois21}}). Another application of cubical edge-labelings can be found in \cite{Fioravanti24} which is concerned with the automorphism groups of right-angled Artin groups. While it is important to distinguish the automorphism group of a RAAG from the automorphism group of its Salvetti-complex (neither of which contains the other in general), it is perhaps not surprising that cubical edge-labelings are a useful tool to study either.

Theorem \ref{thm:characterising.existence.of.cubical.edge.labelings} can also be thought of as a continuation of the work in \cite{CrispWiest04}, where similar ideas of edge-labelings were used to embed graph braid groups into right-angled Artin groups. The idea of such embeddings leads directly to the original definition of special cube complexes \cite{HaglundWise07}. Theorem \ref{thm:characterising.existence.of.cubical.edge.labelings} shows that the relation between cubical edge-labelings and specialness persists into current terminology.
\\

Suppose now we have two locally finite $\CAT$ cube complexes $X$ and $Y$ with cocompact cubical actions $\Lambda \curvearrowright X$ and $\Lambda' \curvearrowright Y$ such that there exist $\Lambda$- and $\Lambda'$-invariant cubical edge-labelings on $X$ and $Y$ respectively. Let $\Sigma$ and $\Sigma'$ denote the set of labels appearing in $X$ and $Y$ respectively. We define the set $\Sigma'^{edge}_{v}$ in analogy to the set $\Sigma^{edge}_v$ above.

Let $g : X \rightarrow Y$ be a cubical isometry and consider the rooted tree isomorphism $T_o^{edge} \rightarrow T'^{edge}_{g(o)}$ that it induces. This rooted tree isomorphism can be encoded as a portrait $(\pre{v}{\sigma(g)})_{v \in V(T_o^{edge})}$. We say that a portrait {\it induces a cubical isometry} if it is induced by a cubical isometry. The fundamental result about cubical isometries and portraits is Theorem \ref{thm:characterization.of.cubical.portraits} which gives a characterisation of portraits that induce cubical isometries. To state this result, we need some additional definitions. We give simplified definitions here in the introduction. They can be found in their full generality in Section \ref{sec:describing.cubical.isometries}.

\begin{definition} 
	Let $s, t \in \Sigma$. We say that $s$ and $t$ {\it commute} if there exists a vertex $o$ such that $s$ and $t$ are labels of outgoing edges at $o$ and these edges span a square. We write $[s,t] = 1$ if $s$ and $t$ commute and $[s,t] \neq 1$ if they do not commute. 
\end{definition}

By property (3) of cubical edge-labelings, commutation does not depend on the vertex $o$. For our next definition, we need the following notation: Given a vertex $v$ in $T_o^{edge}$, let $v^{\circ}$ denote the endpoint of the edge-path in $X$ corresponding to $v$.

\begin{definition}
	
	A portrait $(\pre{v}{\sigma})_{v \in V(T_o^{edge})}$ on $T_o^{edge}$ is called a {\it cubical portrait on $T_o^{edge}$} if it satisfies the following conditions:
	\begin{enumerate}
		\item[$tree$] For all $v \in V(T_o^{edge})$, we have
		\[ \pre{v}{\sigma}( \Sigma_v^{edge} ) = \Sigma'^{edge}_{ \sigma(v) },  \]
		
		\item[$comm$] For all $v \in V(T_o^{edge})$, the bijection $\pre{v}{\sigma} : \Sigma_v^{edge} \rightarrow \Sigma'^{edge}_{\sigma(v)}$ preserves commutation, that is $[s,t] = 1 \Leftrightarrow [\pre{v}{\sigma}(s), \pre{v}{\sigma}(t)] = 1$ for all $s, t \in \Sigma_v^{edge}$. (We write $[s,t] = [\pre{v}{\sigma}(s), \pre{v}{\sigma}(t)]$.)
		
		\item[$par$] For all $v \in V(T_o^{edge})$ and all $s, t \in \Sigma_v^{edge}$ such that $[s,t] = 1$, we have
		\[ \pre{v}{\sigma}(t) = \pre{vs}{\sigma}(t). \]
		
		\item[$inv$] For all $v \in V(T_o^{edge})$ and all $s \in \Sigma_v^{edge}$, we have
		\[ \pre{vs}{\sigma}(s^{-1}) = \pre{v}{\sigma}(s)^{-1}. \]
		
		\item[$end$] For all $v, w \in V(T_o^{edge})$ such that $v^{\circ} = w^{\circ}$, we have $\sigma_v = \sigma_w$.
	\end{enumerate}
\end{definition}

Our fundamental result, Theorem \ref{thm:characterization.of.cubical.portraits}, states that a portrait on $T_o^{edge}$ induces a cubical isometry if and only if it is cubical. Therefore, we can construct cubical isometries by producing cubical portraits on $T_o^{edge}$. In fact, as we will see in Theorem \ref{thm:canonical.portrait.extension}, there is a self-similar subtree $T_o^{red} \subset T_o^{edge}$ such that it is sufficient to define a cubical portrait on the vertices of $T_o^{red}$. This allows us to do some very explicit constructions that have been used to prove a variety of results. We see one application of this in section \ref{sec:topological.generators.for.stabilizers}, where we construct a topologically generating set of the stabilizer group $\Stab_{\Aut(X)}(o)$ and a finite topologically generating set of $\Aut(X)$ (see Theorem \ref{thm:choice.of.A.generates.stabilizer} and Theorem \ref{thm:fin.top.gen.set}).

	

\subsection{Related works}

There are various classes of polyhedral complexes for which more or less expansive theories of their automorphism groups have been developed. There are three that are important to mention here in relation to $\CAT$ cube complexes with special, compact quotient. First, there is the theory of automorphisms on `sufficiently symmetric' trees. (Sufficiently symmetric usually means that the tree can be described in terms of a finite graph, perhaps with some additional information.) Some of the earliest work on this is due to Tits in \cite{Tits70}, which serves as the basis for many of the works on automorphisms of combinatorial complexes. More recently, \cite{ReidSmith22} studied tree-automorphisms from the viewpoint of local actions and developed a theory closely related to the one developed in \cite{HartnickMedici23b} for the study of cubical isometries. The key difference between the two theories is that Reid and Smith work with trees obtained as universal coverings of finite unoriented graphs, while Hartnick and the author work with self-similar trees, which can be viewed as universal coverings of finite {\bf oriented} graphs. The author suspects that the two theories can be unified. However, since the automorphisms of cube complexes require the tracking of oriented edges, the author has decided to focus on self-similar trees in this article.

Second, there is the theory of automorphisms of buildings. Like the theory on cubical isometries, this theory has grown out of the work of Tits. The two theories overlap on right-angled buildings, which are both buildings and $\CAT$ cube complexes (and automorphisms with respect to either structure are automorphisms with respect to both). There is a significant body of literature on the subject with the following two being perhaps the most illuminating regarding the connection to the theory of cubical isometries \cite{Caprace14, CapracedeMedts24}. In particular, a comparison of \cite{CapracedeMedts24} with \cite{HartnickMedici23b} raises some interesting questions about simple normal subgroups.

Third, there has been an investigation of automorphisms of graphs that admit a geometric action by a graphically discrete group \cite{MSSW23}. Since any graph-automorphism on the $1$-skeleton of a $\CAT$ cube complex has a unique extension to a cubical automorphisms, this work interacts with the work presented here whenever we consider a graph that is the $1$-skeleton of a $\CAT$ cube complex with special, compact quotient. In particular, the theories of automorphisms of buildings, graphs with a geometric action, and $\CAT$ cube complexes with special, compact quotient all overlap on the case of Salvetti-complexes. Indeed, if $X$ is a Salvetti-complex, all three theories describe the same automorphism group and address closely related questions, sometimes even recovering exactly the same results.




\subsection*{Structure of paper}

Section \ref{sec:trees} discusses self-similar trees and their isomorphisms. Section \ref{sec:describing.cubical.isometries} introduces cubical edge-labelings and cubical portraits and explains how they encode cubical isometries. In section \ref{sec:cubical.labeling.equivalent.to.special}, we show that $\Lambda$-invariant cubical edge-labelings exist if and only if the action of $\Lambda$ on $X$ is special, free, and cocompact. In section \ref{sec:topological.generators.for.stabilizers}, we prove an important result on how to obtain a set of cubical isometries that generates a dense subgroup of the stabilizer of a vertex. We then use this result to obtain a finite topologically generating set of $\Aut(X)$. Finally, some of the longer and technical proofs are postponed to the appendix.

\subsection*{Acknowledgments}

The author was supported by the FWF grant 10.55776 /ESP124. The author thanks Pierre-Emmanuel Caprace and Anthony Genevois for their comments on an earlier version of this paper.




\section{Trees} \label{sec:trees}




\subsection{Self-similar trees} \label{subsec:self.similar.trees}

Let $A$ be a finite, oriented graph, that is, all its edges have an orientation. We allow edges from a vertex to itself and multiple edges between the same vertices. Let $o$ be a marked vertex in $A$ that we call the {\it starting vertex}. We denote the set of edges in $A$ by $\mathcal{E}(A)$ and the set of outgoing edges at a vertex $v$ by $\mathcal{E}(A)_v$. Let $\Sigma$ be a finite set and let $\ell : \mathcal{E}(A) \rightarrow \Sigma$ be a surjective map such that for every vertex $v$, the restriction $\ell \vert_{\mathcal{E}(A)_v}$ is injective.

Let $\Sigma^*$ denote the set of all words over the alphabet $\Sigma$, including the empty word $\epsilon$. We define $\mathcal{L}_A \subset \Sigma^*$ to be the language of all words that are spelled out by some path of oriented edges that starts at $o$. We set the convention that the path that simply stays at $o$ spells out the empty word $\epsilon$, which is thus contained in $\mathcal{L}_A$. (In other words, we view the triple $(A, o, \ell)$ as a finite automaton where all vertices are accepting states. The language of all words accepted by this automaton is $\mathcal{L}_A$.) Since the restrictions $\ell \vert_{\mathcal{E}(A)_v}$ are all injective, distinct edge-paths starting at $o$ spell out distinct words. In particular, there is a canonical bijection between the set of edge-paths that start at $o$ and the language $\mathcal{L}_A$.

The language $\mathcal{L}_A$ carries a natural structure of a rooted tree $T_A$ which can be obtained as follows: The vertices of $T_A$ are the elements of $\mathcal{L}_A$. Its root is the vertex $\epsilon \in \mathcal{L}_A$. Finally, two vertices $v, w \in \mathcal{L}_A$ are connected by an edge if and only if there exists $s \in \Sigma$ such that either $v = ws$ or $w = vs$. We call $v$ the {\it parent vertex} of $vs$ (and, following this terminology, we obtain {\it ancestor, descendant} and {\it child vertices}). One easily verifies that the resulting graph is a locally finite tree. This tree is also called the {\it path language tree}. The following definition is from \cite{BelkBleakMatucci21}.

\begin{definition}
	A rooted tree $T$ is called {\it self-similar} if there exists a triple $(A, o, \ell)$ such that $T = T_A$.
\end{definition}

The triple $(A, o, \ell)$ induces an edge-labeling on $T_A$, where we label the edge between two vertices $v$ and $vs$ by $s$. We denote this edge-labeling by $\ell$ as well. Note that different triples $(A, o, \ell)$ may induce the same tree, but they might induce different edge labelings. Given a self-similar tree $T$, we usually choose some triple $(A, o, \ell)$ such that $T = T_A$ and use the edge-labeling on $T_A$ as a technical tool.

There is a natural projection map $p : T_A \rightarrow A$ that is obtained as follows: Identify a vertex $v$ in $T_A$ with the corresponding edge-path in $A$ that starts at $o$. We send $v$ to the endpoint of that path. One easily checks that this map on vertices extends to a map between graphs.

For every $v \in \mathcal{L}_A$, we define
\[ \Sigma_v := \{ s \in \Sigma \vert vs \in \mathcal{L}_A \}. \]
Equivalently, $\Sigma_v$ is the set of all labels in $\Sigma$ that appear on an outgoing edge at $p(v)$. For any $v \in \Sigma^* \setminus \mathcal{L}_A$, we define $\Sigma_v := \emptyset$. Similarly, we define
\[ \mathcal{L}_{A, v} := \{ u \in \Sigma^* \vert vu \in \mathcal{L}_A \}. \]
It is an easy exercise that, if $p(v) = p(w)$, the $\Sigma_v = \Sigma_w$ and $\mathcal{L}_{A,v} = \mathcal{L}_{A,w}$. Thus, there are only finitely many languages $\mathcal{L}_{A,v}$.




\subsection{Isomorphisms between self-similar trees} \label{subsec:isomorphisms.between.self.similar.trees}

Let $T_A$ and $T'_A$ be two self-similar trees induced by triples $(A, o, \ell)$ and $(A', o', \ell')$. Let $\Sigma$ and $\Sigma'$ denote the set of labels of $\ell$ and $\ell'$ respectively. (We also define $\Sigma'_v$ etc.\,in analogy to the notation defined before.) The goal of this section is to encode rooted tree-isomorphisms $T_A \rightarrow T'_A$ in terms of a family of bijections between finite sets. {\bf We highlight that, while $T_A$ and $T'_A$ both come with an edge-labeling, we are considering isomorphisms of rooted trees, that is, they do not need to be compatible with the edge-labelings on $T_A$ and $T'_A$ in any way.} Let $g : T_A \rightarrow T'_A$ be an injective rooted tree-morphism. Since $T_A$ and $T'_A$ are trees, $g$ is fully determined by its restriction to the vertices of these trees. Therefore, we may identify $g$ with an injection $g : \mathcal{L}_A \rightarrow \mathcal{L}_{A'}$ that determines $g$ uniquely.

We now define injections $\pre{v}{\sigma(g)} : \Sigma_v \rightarrow \Sigma'_{g(v)}$ by the equation
\[ g(vs) = g(v) \pre{v}{\sigma(g)}(s), \quad \forall v \in \mathcal{L}_A, \quad \forall s \in \Sigma_v. \]
Since $g$ is injective, these maps are injective and they provide us with the following formula: For every $s_1 \dots s_n \in \mathcal{L}_A$, we have
\begin{equation} \label{eq:main.formula}
	g(s_1 \dots s_n) = \pre{\epsilon}{\sigma(g)}(s_1) \dots \pre{s_1 \dots s_{i-1}}{\sigma(g)}(s_i) \dots \pre{s_1 \dots s_{n-1}}{\sigma(g)}(s_n).
\end{equation}
Furthermore, if $u_1 \in \mathcal{L}_A$ and $u_2 = t_1 \dots t_l \in \mathcal{L}_{A,u_1}$, then we define
\[ \pre{u_1}{\sigma(g)}(u_2) := \pre{u_1}{\sigma(g)}(t_1) \dots \pre{u_1 t_1 \dots t_{i-1}}{\sigma(g)}(t_i) \dots \pre{u_1 t_1 \dots t_{l-1}}{\sigma(g)}(t_l) \]
and obtain
\begin{equation} \label{eq:main.formula.extended}
	g(u_1 u_2) = g(u_1) \pre{u_1}{\sigma(g)}(u_2).
\end{equation}

Equation (\ref{eq:main.formula}) shows that we can recover the map $g : \mathcal{L}_A \rightarrow \mathcal{L}_{A'}$ from the family of bijections $(\pre{v}{\sigma(g)})_{v \in \mathcal{L}_A}$. We thus want to understand which families of bijections arise from tree-isomorphisms $T_A \rightarrow T_{A'}$.

\begin{definition}
	Let $T_A$ and $T'_A$ be two self-similar trees induced by triples $(A, o, \ell)$ and $(A', o', \ell')$. Let $\Sigma$ and $\Sigma'$ denote the set of labels of $\ell$ and $\ell'$ respectively. A {\it portrait on $T_A$ to $T'_A$} is a family $(\pre{v}{\sigma})_{v \in \mathcal{L}_A}$ of injective maps $\pre{v}{\sigma} : \Sigma_v \hookrightarrow \Sigma'$.
\end{definition}

Given a portrait $(\pre{v}{\sigma})_{v \in \mathcal{L}_A}$, we can define a map
\[ \sigma : \mathcal{L}_A \rightarrow \Sigma'^* \]
\begin{equation} \label{eq:definition.of.sigma}
	\sigma(s_1 \dots s_n ) := \pre{\epsilon}{\sigma}(s_1) \dots \pre{s_1 \dots s_{i-1}}{\sigma} (s_i) \dots \pre{s_1 \dots s_{n-1}}{\sigma} (s_n).
\end{equation}
Note that the image of $\sigma$ is a-priori not contained in $\mathcal{L}_{A'}$. The following theorem tells us when $\sigma$ induces a rooted tree-isomorphism.

\begin{theorem} \label{thm:characterising.tree.isomorphisms}
	Let $(\pre{v}{\sigma})_{v \in \mathcal{L}_A}$ be a portrait. The map $\sigma$ induces an injective morphism of rooted trees $T_A \rightarrow T'_A$ if and only if for every $v \in \Sigma_v$, we have that $\pre{v}{\sigma}( \Sigma_v ) \subseteq \Sigma'_{\sigma(v)}$.
	
	Furthermore, $\sigma$ is an isomorphism if and only if the inclusions $\pre{v}{\sigma}( \Sigma_v ) \subseteq \Sigma'_{\sigma(v)}$ are all equalities.
\end{theorem}

We recall that we defined $\Sigma'_{w} = \emptyset$ for all $w \in \Sigma'^* \setminus \mathcal{L}_{A'}$, making the condition in the theorem well-defined even if $\sigma(v) \notin \mathcal{L}_{A'}$.

\begin{proof}
	Let $(\pre{v}{\sigma})_{v \in \mathcal{L}_A}$ be a portrait. Following the same construction as in the definition of $\pre{v}{\sigma(g)}$ for a rooted tree-morphism, it is clear that, if $\sigma$ is an injective morphism of rooted trees, then $\pre{v}{\sigma}( \Sigma_v ) \subseteq \Sigma'_{\sigma(v)}$. (It is also clear that, if $\sigma$ is an isomorphism, then these inclusions are equalities.)
	
	For the other direction, suppose that $\pre{v}{\sigma}( \Sigma_v ) \subseteq \Sigma'_{\sigma(v)}$ for every $v \in \mathcal{L}_A$. We need to show three things: First that for every $v \in \mathcal{L}_A$, $\sigma(v) \in \mathcal{L}_{A'}$, second that $\sigma$ sends the root of $T_A$ to the root of $T'_A$, and third that $\sigma$ preserves adjacency of vertices. (The third point follows trivially from equation (\ref{eq:definition.of.sigma}) once we have proven the first.)
	
	Let $v \in \mathcal{L}_A$. We use induction over the length of $v$. By definition, $\sigma(\epsilon) = \epsilon \in \mathcal{L}_{A'}$ and thus $\sigma$ sends the root of $T_A$ to the root of $T'_A$. Now suppose $\sigma(w) \in \mathcal{L}_{A'}$ for every $w$ of length at most $n-1$ and let $v = s_1 \dots s_n \in \mathcal{L}_A$. Then
\[ \sigma(v) = \sigma(s_1 \dots s_{n-1}) \pre{s_1 \dots s_{n-1}}{\sigma}(s_n). \]
By induction assumption, $\sigma(s_1 \dots s_{n-1}) \in \mathcal{L}_{A'}$. By assumption on the portrait, $\pre{s_1 \dots s_{n-1}}{\sigma}(s_n) \in \Sigma'_{\sigma(s_1 \dots s_{n-1})}$. Thus, $\sigma(v) \in \mathcal{L}_{A'}$, which finishes the induction. Injectivity of $\sigma$ follows from injectivity of all the $\pre{v}{\sigma}$.\\

	For the final statement of the theorem, suppose that $\pre{v}{\sigma}( \Sigma_v ) = \Sigma'_{\sigma(v)}$ for all $v \in \mathcal{L}_A$. We need to show that $\sigma$ is surjective on vertices. We do so by induction. Suppose all vertices in $T'_A$ of distance $\leq n-1$ from the root $\epsilon$ lie in the image of $\sigma$ (for $n-1 = 0$, this is clear as $\sigma(\epsilon) = \epsilon$). Now let $w = t_1 \dots t_n \in \mathcal{L}_{A'}$. By induction-assumption, there exists $v_{-} = s_1 \dots s_{n-1}$ such that $\sigma(v_{-}) = t_1 \dots t_{n-1}$. Since $t_1 \dots t_{n} \in \mathcal{L}_{A'}$, we have that $t_n \in \Sigma'_{t_1 \dots t_{n-1}}$. By assumption,
	\[ \pre{v_{-}}{\sigma} ( \Sigma_{v_{-}}) = \Sigma'_{t_1 \dots t_{n-1}} \]
	and thus there exists $s_n \in \Sigma_{v_{-}}$ such that
	\[ \sigma(s_1 \dots s_n) = \sigma(v_{-}) \pre{v_{-}}{\sigma}(s_n) = t_1 \dots t_{n-1} t_n = w. \]
	By induction over $n$, it follows that $\sigma$ is surjective, which concludes the proof.
\end{proof}




\section{Describing cubical isometries as tree isomorphisms} \label{sec:describing.cubical.isometries}




\subsection{Edge paths and cubical edge-labelings} \label{subsec:edge.paths}

Let $X$ and $Y$ be locally finite $\CAT$ cube complexes and $\Lambda$ and $\Lambda'$ two finitely generated groups acting cocompactly by cubical isometries on $X$ and $Y$ respectively. For an introduction to the basic terminology and properties of $\CAT$ cube complexes, we refer to \cite{Wise21}. In this section, we show how cubical isometries $X \rightarrow Y$ can be understood as isomorphisms between self-similar trees, provided that $X$ and $Y$ admit a certain type of edge-labeling. This will culminate in a characterisation of those portraits that induce cubical isometries. We begin with the necessary terminology to turn cubical isometries into isomorphisms of self-similar trees.

\begin{remark}
When we considered the graph $A$ in the previous section, all edges came with a chosen orientation. In the $1$-skeleton of a cube complex, edges don't come with a fixed orientation. When we speak of an {\it oriented edge} in $X$, we mean an edge together with a chosen orientation. (In contrast to the situation in $A$, either choice of orientation is allowed here.) Let $\overrightarrow{\mathcal{E}}(X)$ denote the collection of oriented edges in $X$. For any vertex $o \in X^{(0)}$, we write $\overrightarrow{\mathcal{E}}(X)_o$ for the set of all oriented edges starting at $o$. Given an oriented edge $e$, we denote the edge with opposite orientation by $e^{-1}$.
\end{remark}

Let $S$ be a square in $X$. If $e$ and $e'$ are two unoriented edges in $S$, we call them {\it parallel in $S$} if they don't share a vertex. If $e$ and $e'$ are two oriented edges in $S$, we call them {\it parallel in $S$} if their underlying unoriented edges are parallel and their starting points are connected by an edge in $S$. These two notions generate an equivalence relation on $\mathcal{E}(X)$ and $\overrightarrow{\mathcal{E}}(X)$ respectively. We say two (un)oriented edges are parallel, if they are equivalent under this equivalence relation. We call the equivalence classes {\it (oriented) parallel classes}.

\begin{definition} \label{def:cubical.edge.labeling}
	A {\it cubical edge-labeling} on $X$ is a surjective map $\ell : \overrightarrow{\mathcal{E}}(X) \rightarrow \Sigma$, where $\Sigma$ is a finite set, satisfying the following properties:
	\begin{enumerate}
		\item For every vertex $o \in X^{(0)}$, the restriction $\ell \vert_{\overrightarrow{\mathcal{E}}(X)_o}$ is injective.
		
		\item If $e, e' \in \overrightarrow{\mathcal{E}}(X)$ are parallel oriented edges, then $\ell(e) = \ell(e')$.
				
		\item Let $o, o' \in X^{(0)}$, let $e, f \in \overrightarrow{\mathcal{E}}(X)_o$, and let $e', f' \in \overrightarrow{\mathcal{E}}(X)_{o'}$ such that $\ell(e) = \ell(e')$ and $\ell(f) = \ell(f')$. Then $e$ and $f$ span a square if and only if $e'$ and $f'$ do.
		
		\item There exists a fixed-point free involution $\cdot^{-1} : \Sigma \rightarrow \Sigma$ such that for every $e \in \overrightarrow{\mathcal{E}}(X)$, $\ell(e^{-1}) = \ell(e)^{-1}$.

	\end{enumerate}
	
	We call a cubical edge-labeling {\it $\Lambda$-invariant} if for all $g \in \Lambda$ and all $e \in \overrightarrow{\mathcal{E}}(X)$, we have $\ell(e) = \ell(g(e))$.
\end{definition}

We will characterise the cube complexes that admit a $\Lambda$-invariant cubical edge-labeling in Theorem \ref{thm:characterising.existence.of.cubical.edge.labelings}. For now, we simply work under the assumption that $X$ admits such a labeling.

We fix once and for all a $\Lambda$-invariant cubical edge-labeling $\ell$ on $X$ and write $\Sigma$ for the set of labels. This labeling allows us to associate to every edge path in $X$ a word over the alphabet $\Sigma$ by consecutively reading the labels of the oriented edges in the edge path. Given an edge path $\gamma$ in $X$, we write $\ell(\gamma)$ for this associated word. Due to property (1) of cubical edge-labelings, one easily sees that an edge path $\gamma$ is uniquely determined by its starting vertex and the word $\ell(\gamma)$. We obtain an injective map
\[ \ell : \{ \text{edge paths that start at } o \} \rightarrow \Sigma^*. \]
We denote the image of this map by $\mathcal{L}_o^{edge}$. Since $\ell$ is a bijection onto $\mathcal{L}_o^{edge}$, we will identify edge paths in $X$ that start at $o$ with their corresponding word in $\mathcal{L}_o^{edge}$. In particular, we will treat elements of $\mathcal{L}_o^{edge}$ both as words over $\Sigma$ and as edge paths that have a starting point and an endpoint.

\begin{notation} \label{not:edge}
	For every $v \in \mathcal{L}_o^{edge}$, we denote its endpoint by $v^{\circ} \in X^{(0)}$. In addition, we write
	\[ \mathcal{L}_v^{edge} := \mathcal{L}_{v^{\circ}}^{edge} = \{ u \in \Sigma^* \vert vu \in \mathcal{L}_o^{edge} \}, \]
	and we write
	\[ \Sigma_v^{edge} := \Sigma_{v^{\circ}}^{edge} := \ell( \overrightarrow{ \mathcal{E} }(X)_{v^{\circ}} ) \]
	for the collection of of labels on outgoing edges at the endpoint of $v$. If $v \in \Sigma^* \setminus \mathcal{L}_o^{edge}$, we define $\Sigma_v^{edge} := \emptyset$. It is an easy exercise that $\mathcal{L}^{edge}_{v}$ depends only on the endpoint of $v$.
\end{notation}

\begin{definition} \label{def:commuting.labels}
	Let $s, t \in \Sigma$. We say that $s$ and $t$ {\it commute} if there exists a vertex $p$ such that $s, t \in \Sigma_p^{edge}$ and the outgoing edges at $p$ labeled by $s$ and $t$ span a square. We write $[s,t] = 1$ if $s$ and $t$ commute and $[s,t] \neq 1$ if they do not commute.
\end{definition}

By condition (3) of a cubical edge-labeling, if there are outgoing edges labeled $s$ and $t$ at two different vertices $p$ and $p'$, then the edges at $p$ span a square if and only if the edges at $p'$ do. Thus, commutation of labels does not depend on the vertex $p$. We highlight that for any $s \in \Sigma$, we have $[s,s] \neq 1$ and $[s, s^{-1}] \neq 1$. (If not, one could produce a square with a vertex that has two distinct outgoing edges that are both labeled by $s$, a contradiction to condition (1) of a cubical edge-labeling.) 
Given two pairs of labels $(s,t)$ and $(s', t')$, we write $[s,t] = [s',t']$ whenever $s$ and $t$ commute if and only if $s'$ and $t'$ commute; for example, $[s,t] = [s,t^{-1}]$ for all $s, t \in \Sigma$, which follows from the properties of cubical edge-labelings.

We highlight that, if $s, t \in \Sigma_v^{edge}$ and $[s,t] = 1$, then $t \in \Sigma_{vs}^{edge}$ and $s \in \Sigma_{vt}^{edge}$ by condition (2) of cubical isometries. This becomes essential when we have, for example, a label $s \in \Sigma_o^{edge}$ and a word $tt' \in \mathcal{L}_o^{edge}$ such that $[s,t] = [s,t'] = 1$. In that case, the words $stt'$, $tst'$, and $tt's$ all lie in $\mathcal{L}_o^{edge}$ and $stt'^{\circ} = tst'^{\circ} = tt's^{\circ}$. More formally, we have the following definition.

\begin{definition} \label{def:commuting.words}
	Let $o$ be a vertex and let $v, w \in \mathcal{L}_o^{edge}$. We say that $v$ and $w$ commute if and only if for any letter $s$ in $v$ and any letter $t$ in $w$, we have $[s,t] = 1$. We write $[v,w] = 1$ and $[v,w] \neq 1$ if $v$ and $w$ commute or do not commute respectively.
\end{definition}

As indicated above, when $[v,w] = 1$, then we can splice together the words $v$ and $w$ in any way and obtain an element in $\mathcal{L}_o^{edge}$.

\begin{definition} \label{def:reduced.words}
	An edge-path is called {\it reduced} it and only if it crosses no hyperplane more than once. (Such paths are also called {\it combinatorial geodesics}). A word $v \in \mathcal{L}_o^{edge}$ is called {\it reduced} if and only if its corresponding edge path is reduced. We call $v$ {\it reducible} if and only if it is not reduced.
\end{definition}

\begin{notation} \label{not:red}
	We denote the image under $\ell$ of all reduced edge paths starting at $o$ by $\mathcal{L}_o^{red} \subset \mathcal{L}_o^{edge}$. Furthermore, we write for every $v \in \mathcal{L}^{red}_o$
	\[ \mathcal{L}_v^{red} := \{ u \in \Sigma^* \vert vu \in \mathcal{L}_o^{red} \}, \]
	\[ \Sigma_v^{red} := \{ s \in \Sigma_v^{edge} \vert vs \in \mathcal{L}_o^{red} \}. \]
	In addition, we write
	\[ \reduc(v) := \Sigma_v^{edge}Ê\setminus \Sigma_v^{red}, \]
	\[Ê\langle \reduc(v) \rangle := \reduc(v) \cup \{Êt \in \Sigma_v^{edge} \vert \exists s \in \reduc(v) : [s,t] = 1 \}. \]
	It is a standard exercise about $\CAT$ cube complexes that all elements in $\reduc(v)$ commute pairwise. (Hint: Every $s \in \reduc(v)$ corresponds to an edge at $v^{\circ}$ that crosses a hyperplane separating $o$ from $v^{\circ}$. Show that these hyperplanes intersect pairwise.)
\end{notation}

\begin{definition} \label{def:innermost.cancellation}
	Let $s_1 \dots s_n \in \mathcal{L}_o^{edge}$ and let $1 \leq i < j \leq n$. We say the pair $s_i, s_j$ is an {\it innermost cancellation} if $s_i \dots s_j$ is a reducible word while $s_i \dots s_{j-1}$ and $s_{i+1} \dots s_j$ are both reduced.
\end{definition}

One immediately sees that, if $s_i$, $s_j$ is an innermost cancellation, then the $i$-th and $j$-th edge in the edge-path that starts at $o$ and corresponds to $s_1 \dots s_n$ cross the same hyperplane in opposite directions. Furthermore, any reducible word contains an innermost cancellation. The following Lemma is a standard fact about reducible edge-paths in $\CAT$ cube complexes. A proof using the language of labels can be found in {\cite[Lemma 2.21]{HartnickMedici23b}} and a visualisation is given in Figure \ref{fig:innermost.cancellation}.

\begin{figure}
	\centering
    	\def\svgwidth{5in}
	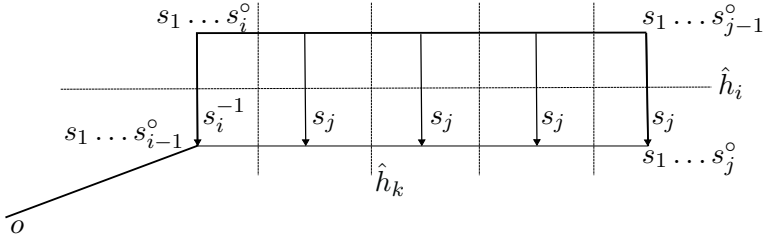
	\caption{The pair $s_i, s_j$ depicted here is an innermost cancellation of the word/edge-path $s_1 \dots s_n$. One can inductively show that the hyperplanes $\hat{h}_k$ for $i < k < j$ all intersect with the hyperplane $\hat{h}_i$. This yields a chain of squares as depicted in the figure. From this chain of squares, we can conclude that $s_j = s_i^{-1}$ and that $[s_i, s_k] = 1$ for all $i < k < j$.}
	\label{fig:innermost.cancellation}
\end{figure}

\begin{lemma} \label{lem:innermost.cancellation.reduction}
	Let $s_1 \dots s_n \in \mathcal{L}_o^{edge}$ such that $s_i$ and $s_j$ are an innermost cancellation. Then $[s_i, s_k] = [s_k, s_j] = 1$ for every $i < k < j$. Furthermore,
	\[ s_1 \dots s_n^{\circ} =  s_1 \dots \hat{s_i} \dots \hat{s_j} \dots s_n^{\circ}, \]
	where $\hat{s_i}$ indicates that we skipped the letter $s_i$.
\end{lemma}

The following Lemma is a variation of {\cite[Proposition 2.15]{HartnickMedici23b}}.

\begin{lemma} \label{lem:regular.language}
	Both $\mathcal{L}_o^{edge}$ and $\mathcal{L}_o^{red}$ admit a triple $(A^{edge}, o, \ell)$ and $(A^{red}, o, \ell)$ respectively such that they are the respective language induced by this triple.
\end{lemma}

\begin{proof}
	We start with $\mathcal{L}_o^{edge}$. Consider the quotient of $X$ by the action of $\Lambda$ and denote it by $S$. By assumption, this is a compact cube complex. Let $o'$ be the image of $o$ under the projection to the quotient. We create $A^{edge}$ by taking the $1$-skeleton of $S$ and replacing every unoriented edge by two oriented edges that connect the endpoints of the edge in both directions. The vertex $o'$ will be the root. Since the cubical edge-labeling on $X$ is $\Lambda$-invariant, it descends to $S$ and induces an edge-labeling (again denoted by $\ell$) on the oriented edges in the new graph. One easily verifies that the triple $(A^{edge}, o', \ell)$ induces the language $\mathcal{L}_o^{edge}$.
	
	The construction of $A^{red}$ is analogous to constructions in \cite{NibloReeves98, HartnickMedici23b}. Let $v, w \in \mathcal{L}_o^{red}$. We write $v \sim w$ if and only if $v^{\circ} \in \Lambda \cdot w^{\circ}$ and $\reduc(v) = \reduc(w)$. This defines an equivalence relation. Since $\Lambda$ acts cocompactly on $X$ and $\Sigma$ is finite, there are only finitely many equivalence classes. We denote the equivalence class of a word $v$ by $[v]$. Note that the empty word $\epsilon$ is the unique word such that $\reduc(\epsilon) = \emptyset$ and thus it forms its own equivalence class. We define the graph $A^{red}$ as follows: Its vertices are the equivalence classes $[v]$. For every $s \in \Sigma_{v}^{red}$, we introduce an edge from $[v]$ to $[vs]$ labeled by $s$.
	
	We need to show that $[vs]$ does not depend on the choice of representative of $[v]$. Suppose $v \sim v'$. Then there exists $g \in \Lambda$ such that $g(v^{\circ}) = v'^{\circ}$ and $\reduc(v) = \reduc(v')$. Therefore, $\Sigma_v^{red} = \Sigma_{v'}^{red}$ and thus $v's \in \mathcal{L}_o^{red}$. Furthermore, since the labeling on $X$ is $\Lambda$-invariant, we conclude that $g$ sends the edge from $v^{\circ}$ to $vs^{\circ}$ to the unique outgoing edge at $v'^{\circ}$ that is labeled by $s$. Therefore, $g$ sends $vs^{\circ}$ to $v's^{\circ}$. We are left to show that $\reduc(vs) = \reduc(v's)$. It is a standard exercise about cube complexes and cubical edge-labelings that
	\[ \reduc(vs) = \{ s^{-1} \} \cup \{ t \in \reduc(v) \vert [s,t] = 1 \}. \]
	Since $\reduc(v) = \reduc(v')$, this implies that $\reduc(vs) = \reduc(v's)$. This implies that $vs \sim v's$ and thus our edges are well-defined.
	
	This construction produces a triple $(A^{red}, [\epsilon], \ell)$ that forms an oriented, labeled graph with a marked vertex. By construction, the language induced by this graph is exactly $\mathcal{L}_o^{red}$.
\end{proof}

Since $\mathcal{L}_o^{edge}$ and $\mathcal{L}_o^{red}$ are defined by a rooted, oriented, labeled graph $(A^{edge}, o', \ell)$ and $(A^{red}, o', \ell)$ respectively, there exist self-similar trees $T_o^{edge}$ and $T_o^{red}$ as discussed in section \ref{subsec:self.similar.trees}. These trees inherit a labeling from the labeling on $A^{edge}$ and $A^{red}$ respectively. Since $\mathcal{L}_o^{red} \subset \mathcal{L}_o^{edge}$, there is a canonical, label-preserving embedding $T_o^{red} \hookrightarrow T_o^{edge}$.

Note that there is a canonical map $T_o^{edge} \rightarrow X^{(1)}$ that sends every vertex of $T_o^{red}$, which is a word $v \in \mathcal{L}_o^{red}$, to the vertex $v^{\circ}$ in $X$. Note that this map is not injective unless $X$ is a tree. It will frequently be useful to think of $T_o^{edge}$ as sitting inside of $X$ in this way as it provides valuable geometric intuition.

\begin{remark} \label{rem:root.of.subtree}
	Let $T'$ be a rooted tree and $T$ any subtree of $T'$. There exists a unique vertex $v$ in $T$ that is the closest to the root of $T'$. We view $v$ as the induced root of $T$ as a subtree of $T'$.
\end{remark}




\subsection{Cubical portraits} \label{subsec:cubical.portraits}

\begin{notation}
	Let $X$ and $Y$ be two locally finite $\CAT$ cube complexes with cocompact cubical actions $\Lambda \curvearrowright X$ and $\Lambda' \curvearrowright Y$. Throughout this section, we assume that there exist $\Lambda$- and $\Lambda'$-invariant cubical edge-labelings $\ell$ and $\ell'$ on $X$ and $Y$ respectively. Let $o$ and $o'$ be vertices in $X$ and $X'$ respectively. In $X$ we obtain the established notation of $\mathcal{L}_o^{edge}$, $\mathcal{L}_o^{red}$, $T_o^{edge}$, $T_o^{red}$, $\Sigma_v^{edge}$ etc. We analogously define notation $\mathcal{L}'^{edge}_{o'}$, $T'^{edge}_{o'}$, $\Sigma'^{edge}_w$ etc.\,for $Y$.
\end{notation}

\begin{definition} \label{def:portrait.on.subtree}
	Let $v_0 \in \mathcal{L}_o^{edge}$, $w_0 \in \mathcal{L}'^{edge}_{o'}$, and let $T$ be a subtree of $T_o^{edge}$ with root $v_0$. A {\it portrait on $T$ to $w_0$ in $T'^{edge}_{o'}$} is a collection $(\pre{v}{\sigma})_{v \in T}$ of injective maps $\pre{v}{\sigma} : \Sigma_v^{edge} \hookrightarrow \Sigma'$ such that $\pre{v_0}{\sigma}( \Sigma_{v_0}^{edge} ) = \Sigma'^{edge}_{w_0}$.
	
	If $w_0$ and $T'^{edge}_{o'}$ are clear from context, we simply speak of a portrait on $T$.
\end{definition}

\begin{notation} \label{not:portrait.formula}
	Let $(\pre{v}{\sigma})_{v \in T}$ be a portrait on $T$ to $w_0$. For every $v = v_0 s_1 \dots s_n \in \mathcal{L}_o^{edge}$, we define
	\[ \pre{v_0}{\sigma}(s_1 \dots s_n) := \pre{v_0}{\sigma}(s_1) \pre{v_0 s_1}{\sigma}(s_2) \dots \pre{v_0 s_1 \dots s_{i-1}}{\sigma}(s_i) \dots \pre{v_0 s_1 \dots s_{n-1}}{\sigma}(s_n),  \]
	\[ \sigma(v) = \sigma(v_0 s_1 \dots s_n) := w_0 \pre{v_0}{\sigma}(s_1 \dots s_n). \]
	In general, the expressions on the right-hand-side in these definitions are simply words over the alphabet $\Sigma'$. They may lie outside of $\mathcal{L}'^{red}_{o'}$ or $\mathcal{L}'^{edge}_{o'}$.
\end{notation}

\begin{definition} \label{def:cubical.portrait.on.subtree}	
	A portrait $(\pre{v}{\sigma})_{v \in T}$ on $T$ to $w_0$ in $T'^{edge}_{o'}$ is called a {\it cubical portrait on $T$ to $w_0$ in $T'^{edge}_{o'}$} if it satisfies the following conditions:
	\begin{enumerate}
		\item[$tree$] For all $v = v_0 s_1 \dots s_n \in T$, we have
		\[ \pre{v}{\sigma}( \Sigma_v^{edge} ) = \Sigma'^{edge}_{ \sigma(v) },  \]
		
		\item[$comm$] For all $v \in T$, the bijection $\pre{v}{\sigma} : \Sigma_v^{edge} \rightarrow \Sigma'^{edge}_{\sigma(v)}$ preserves commutation, that is $[s,t] = 1 \Leftrightarrow [\pre{v}{\sigma}(s), \pre{v}{\sigma}(t)] = 1$ for all $s, t \in \Sigma_v^{edge}$.
		
		\item[$par$] For all $v \in T$, all $s \in \Sigma_v^{edge}$ such that $vs \in T$, and all $t \in \Sigma_v^{edge}$ such that $[s,t] = 1$, we have
		\[ \pre{v}{\sigma}(t) = \pre{vs}{\sigma}(t). \]
		
		\item[$inv$] For all $v \in T$ and all $s \in \Sigma_v^{edge}$ such that $vs \in T$, we have
		\[ \pre{vs}{\sigma}(s^{-1}) = \pre{v}{\sigma}(s)^{-1}. \]
		
		\item[$end$] For all $v, w \in T$ such that $v^{\circ} = w^{\circ}$, we have $\pre{v}{\sigma} = \pre{w}{\sigma}$.
	\end{enumerate}
\end{definition}

Our main results concern the case where $T = T_o^{edge}$ or $T = T_o^{red}$. However, the generality of the definition above is useful when it comes to explicit constructions.\\

Let $g : X \rightarrow Y$ be a cubical isometry. This isometry induces an automorphism of rooted trees $T_o^{edge} \rightarrow T'^{edge}_{o'}$ that restricts to a rooted tree automorphism $T_o^{red} \rightarrow T'^{red}_{o'}$. This rooted tree automorphism can be encoded as the following portrait on $T_o^{edge}$ to $\epsilon$ in $T'^{edge}_{o'}$: At every vertex $w \in \mathcal{L}_o^{edge}$, we define $\pre{w}{\sigma(g)} : \Sigma_w^{edge} \rightarrow \Sigma'^{edge}_{g(w)}$ to be the unique bijection that satisfies
\[ g(ws) = g(w) \pre{w}{\sigma(g)}(s) \]
for every $s \in \Sigma_w^{edge}$. The family $(\pre{w}{\sigma(g)})_{w \in T_o^{edge} }$ defines a portrait on $T_o^{edge}$. We call the maps $\pre{w}{\sigma(g)}$ the {\it local actions} of $g$. We say that a portrait {\it defines a cubical isometry} on $X$ if it can be obtained from a cubical isometry $g \in \Aut(X)$ in this way. The following is a generalisation of {\cite[Theorem J]{HartnickMedici23b}}.

\begin{theorem} \label{thm:characterization.of.cubical.portraits}
	A portrait on $T_o^{edge}$ to $\epsilon$ in $T'^{edge}_{o'}$ is cubical if and only if it defines a cubical isometry from $X$ to $Y$ that sends $o$ to $o'$.
	
	In particular, the stabilizer group $\Stab_{\Aut(X)}(o)$ is in $1-1$ correspondence with the set of all cubical portraits on $T_o^{edge}$ to $\epsilon$ in $T_{o}^{edge}$.
\end{theorem}

\begin{remark} \label{rem:basics.on.cubical.conditions}
	The following are some properties of the five conditions $tree$, $comm$, $par$, $inv$, $end$ that are useful when studying cubical portraits on $T_o^{edge}$ and its subtrees.
	\begin{itemize}
		\item	The condition $comm$ can be checked for every vertex $v \in T$ individually. It is a purely local condition.
		
		\item The conditions $par$ and $inv$ are statements about two vertices in $T$ that are connected by an edge. They can be checked locally by considering pairs $v, vs$ of reduced words.
		
		\item In condition $tree$, the set $\Sigma'^{edge}_{\sigma(v)}$ is empty if and only if $\sigma(v) \notin \mathcal{L}_{o'}^{edge}$. Thus, the equalities in $tree$ can only hold if $\sigma(v) \in \mathcal{L}_{o'}^{edge}$, making this an implicit part of condition $tree$. In other words, $tree$ is equivalent to stating that the portrait $(\sigma_v)_{v \in T_o^{edge}}$ induces a rooted tree automorphism $T_o^{edge} \rightarrow T'^{edge}_{o'}$ (see Theorem \ref{thm:characterising.tree.isomorphisms}).
		
		\item The condition $tree$ is not a local condition because of the term $\sigma(v)$. Writing $v = v_0 s_1 \dots s_n$, the expression $\sigma(v)$ depends on the maps $\sigma_{v_0}, \dots \sigma_{v_0 s_1 \dots s_{n-1}}$. In practice, constructions of cubical portraits are frequently designed so that one can show condition $tree$ for words of the form $v_0 s_1 \dots s_n \in T$ via induction on $n$.
				
		\item Condition $end$ is not a local condition. It guarantees that no incompatibilities occur within the portrait when pushing it forward to $X$ under the canonical map $T_o^{edge} \rightarrow X^{(1)}$. Condition $end$ also implies that, if a portrait satisfies $end$ and $(v, vs)$ and $(w, ws)$ are two pairs such that $v^{\circ} = w^{\circ}$, then we need to check $par$ and $inv$ only for one of these two pairs. Similarly, $comm$ only needs to be checked at either $v$ or $w$.
	\end{itemize}
\end{remark}

We postpone the proof of Theorem \ref{thm:characterization.of.cubical.portraits} to the appendix. For now, we just point out a basic property of cubical portraits on $T_o^{edge}$: It is clear from condition $tree$ that a cubical portrait on $T_o^{edge}$ to $\epsilon$ in $T'^{edge}_{o'}$ induces an rooted tree-isomorphism $T_o^{edge} \rightarrow T'^{edge}_{o'}$. The conditions of a cubical portrait imply that this map restricts to a rooted tree-isomorphism $T_o^{red} \rightarrow T'^{red}_{o'}$.

\begin{lemma} \label{lem:cubical.implies.reduction.preserving}
	Let $(\pre{v}{\sigma})_{v \in \mathcal{L}_o^{edge}}$ be a cubical portrait on $T^{edge}_o$ to $\epsilon$. Then for every $v \in \mathcal{L}_o^{red}$, we have that
	\[ \pre{v}{\sigma}( \Sigma_v^{red} ) = \Sigma'^{red}_{\sigma(v)}. \]
\end{lemma}

\begin{proof}
	We first show that $\pre{v}{\sigma}( \Sigma^{red}_v) \subseteq \Sigma'^{red}_{\sigma(v)}$. Let $v \in \mathcal{L}_o^{red}$ and $s \in \Sigma_v^{red}$. Our goal is to show that $\sigma(v) \sigma_v(s)$ is reduced. We do so by showing that $\sigma$ sends reduced paths to reduced paths. We show this by induction over the length of the path. The induction start is clear, since $\sigma(\epsilon) = \epsilon$. Now let $v = s_1 \dots s_n$ and $s \in \Sigma_v^{red}$ and assume by induction that $\sigma(v)$ is reduced. Suppose $\sigma(v) \sigma_v(s)$ is reducible. Then there exists $1 \leq j \leq n$ such that the pair $\sigma_{s_1 \dots s_{j-1}}(s_j)$ and $\sigma_v(s)$ form an innermost cancellation in $\sigma(vs)$. By Lemma \ref{lem:innermost.cancellation.reduction}, we conclude that $\sigma_{s_1 \dots s_{j-1}}(s_j)$ commutes with all letters of the word $\sigma(v)$ that appear after $\sigma_{s_1 \dots s_{j-1}}(s_j)$. By $comm$, this implies that $s_j$ commutes with $s_k$ for every $k > j$ and we obtain a word $v' = s_1 \dots \hat{s_j} \dots s_n$ such that $v^{\circ} = v's_j^{\circ}$. By $par$, we obtain that $\sigma_{s_1 \dots s_{j-1}}(s_j) = \sigma_{v'}(s_j)$. By $end$, we obtain that $\sigma_v(s) = \sigma_{v' s_j}(s)$.
	
	By construction, $\sigma_{v'}(s_j)$ and $\sigma_{v' s_j}(s)$ form an innermost cancellation in the word $\sigma(v' s_j s)$ and thus $\sigma_{v' s_j}(s) = \sigma_{v'}(s_j)^{-1}$. On the other hand, $inv$ implies that $\sigma_{v'}(s_j)^{-1} = \sigma_{v' s_j}(s_j^{-1})$. Since $\sigma_{v's_j}$ is injective, this implies that $s = s_j^{-1}$ and thus $v' s_j s$ is reducible. Therefore, $v s$ is reducible with the pair $(s_j, s)$ being an innermost cancellation. This contradicts the assumption that $s \in \Sigma_{v}^{red}$. We conclude that $\sigma_v( \Sigma_v^{red} ) \subseteq \Sigma'^{red}_{\sigma(v)}$.
	
	Suppose now that $vs$ is reducible, but $\sigma(v) \sigma_v(s)$ is reduced. We find $1 \leq j \leq n$ such that $s_j$ and $s$ form an innermost cancellation of the word $vs$. By an analogous argument as above, we conclude that $\sigma_{s_1 \dots s_{j-1}}(s_j)$ and $\sigma_v(s)$ form an innermost cancellation of the word $\sigma(v) \sigma_v(s)$, contradicting the fact that $\sigma(v) \sigma_v(s)$ is reduced. We conclude that $\sigma_v(\Sigma_v^{red}) = \Sigma'^{red}_{\sigma(v)}$.
\end{proof}

In fact, any cubical portrait on the subtree $T^{red}_o$ has a unique extension to a cubical portrait on $T^{edge}_o$ which thus induces a cubical isometry.

\begin{theorem} \label{thm:canonical.portrait.extension}
	Any cubical portrait $(\pre{v}{\sigma})_{v \in \mathcal{L}_o^{red}}$ on $T_o^{red}$ to $\epsilon$ in $T'^{edge}_{o'}$ has a unique extension to a cubical portrait on $T_o^{edge}$ to $\epsilon$ in $T'^{edge}_{o'}$. In particular, any cubical portrait on $T_o^{red}$ to $\epsilon$ in $T'^{edge}_{o'}$ induces a cubical isometry $X \rightarrow Y$ that sends $o$ to $o'$.
\end{theorem}

We postpone the proof of this theorem to the appendix.




\section{Characterising $\CAT$ cube complexes with cubical edge-labelings} \label{sec:cubical.labeling.equivalent.to.special}

In this section, we characterise the $\CAT$ cube complexes that admit a $\Lambda$-invariant cubical edge-labeling. This is the content of the following theorem, which is a reformulation of {\cite[Theorem A]{HartnickMedici23b}}.

\begin{theorem} \label{thm:characterising.existence.of.cubical.edge.labelings}
	Let $X$ be a $\CAT$ cube complex, $\Lambda$ a group acting cocompactly by cubical isometries on $X$, and $S$ the quotient of $X$ under the action of $\Lambda$. Then $X$ admits a $\Lambda$-invariant cubical edge-labeling if and only if the action $\Lambda \curvearrowright X$ is free and $S$ is special.
\end{theorem}

A cube complex is called special, if four pathologies are excluded from the behaviour of its hyperplanes. See {\cite[Section 6.b]{Wise21}} for a precise definition and a visualisation of the four pathologies. We break the proof down into several Lemmas.

\begin{lemma} \label{lem:special.implies.cubical.edge.labeling}
	If $S$ is special and $\Lambda \curvearrowright X$ is free, then $X$ admits a $\Lambda$-invariant cubical edge-labeling.
\end{lemma}

\begin{proof}
	Suppose $S$ is special. We label the oriented edges in $S$ by labeling every oriented parallel class of oriented edges by a unique label. Let $\Sigma$ be the set of these labels. We denote this labeling by $\ell_S : \overrightarrow{\mathcal{E}}(S) \rightarrow \Sigma$. We highlight that two oriented edges have the same label if and only if they are parallel. Let $\cdot^{-1} : \Sigma \rightarrow \Sigma$ be the map that sends the label of any oriented edge to the label of the edge obtained by reversing orientation. This map is well-defined since two oriented edges have the same label if and only if they are parallel and being parallel is preserved under reversal of orientation. By construction, $\cdot^{-1}$ is an involution. Since $S$ is special, every hyperplane $h$ in $S$ is two-sided. Therefore, $\cdot^{-1}$ has no fixed point. 
		
	Since $\Lambda$ acts freely on $X$ and $X$ is a cube complex, the projection map $\pi : X \rightarrow S$ is a covering map. In fact, since $X$ is simply connected, it is the universal covering of $S$. The oriented edge-labeling on $S$ lifts to $X$ and we denote this lift by $\ell : \overrightarrow{\mathcal{E}}(X) \rightarrow \Sigma$. By definition, this lift is $\Lambda$-invariant. We claim that $\ell$ is a cubical edge-labeling.
	
	We start by proving property (1) of cubical edge-labelings. Let $o \in X^{(0)}$ and suppose there are two distinct edges $e, e' \in \overrightarrow{\mathcal{E}}(X)_o$ such that $\ell(e) = \ell(e')$. By construction of $\ell$, this implies that $\ell_S( \pi(e) ) = \ell_S( \pi(e') )$. By construction of $\ell_S$, this implies that $\pi(e)$ and $\pi(e')$ are parallel oriented edges with the same starting point. Furthermore, $\pi(e) \neq \pi(e')$ since $\Lambda$ acts freely and $e$ and $e'$ are distinct edges that start at the same vertex. We conclude that the hyperplane crossed by both $\pi(e)$ and $\pi(e')$ self-osculates. However, since $S$ is special, it cannot contain a self-osculating hyperplane, a contradiction. This proves (1).
	
	Property (2) is an immediate consequence of the fact that parallel oriented edges in $S$ have the same label and $\ell$ is the lift of $\ell_S$.
	
	For property (3), consider vertices $o, o' \in X^{(0)}$ and edges $e, f \in \overrightarrow{\mathcal{E}}(X)_o$ and $e', f' \in \overrightarrow{\mathcal{E}}(X)_{o'}$ such that $\ell(e) = \ell(e')$ and $\ell(f) = \ell(f')$. Consider the projections $\pi(o), \pi(o')$ with outgoing edges $\pi(e), \pi(f)$ and $\pi(e'), \pi(f')$ respectively. Since $\ell$ is the lift of $\ell_S$, we have that $\ell_S(\pi(e)) = \ell_S(\pi(e'))$ and $\ell_S(\pi(f)) = \ell_S(\pi(f'))$. Since two oriented edges in $S$ have the same label if and only if they cross the same hyperplane, we find hyperplanes $\hat{h}$ and $\hat{k}$ such that $\pi(e)$ and $\pi(e')$ cross $\hat{h}$ in the same direction and $\pi(f)$ and $\pi(f')$ cross $\hat{k}$ in the same direction. If $\pi(e)$ and $\pi(f)$ span a square, while $\pi(e')$ and $\pi(f')$ do not, then $\hat{h}$ and $\hat{k}$ inter-osculate, which cannot happen in a special cube complex. Thus, $\pi(e)$ and $\pi(f)$ span a square if and only if $\pi(e')$ and $\pi(f')$ span a square. Since $\pi : X \rightarrow S$ is a covering map, this property lifts to $X$ and we conclude that $e$ and $f$ span a square if and only if $e'$ and $f'$ span a square. 
	
	For property (4), note that, by construction, $\ell_S(e^{-1}) = \ell_S(e)^{-1}$ where we use the definition of $\cdot^{-1}$ introduced in the initial construction. This property lifts to $\ell$ which proves property (4). We conclude that $\ell$ is a $\Lambda$-invariant cubical edge-labeling.
\end{proof}

\begin{lemma}
	If $X$ admits a $\Lambda$-invariant cubical edge-labeling, then $\Lambda$ acts freely on $X$.
\end{lemma}

\begin{proof}
	Let $\ell$ be a $\Lambda$-invariant cubical edge-labeling, $o \in X^{(0)}$, and $\lambda \in \Lambda$ such that $\lambda(o) = o$. Therefore, $\lambda \in G_o$ and we can describe $\lambda$ in terms of the family of local actions $(\pre{v}{\sigma(\lambda)})_{v \in \mathcal{L}^{edge}_o}$. Since $\ell$ is $\Lambda$-invariant, $\lambda$ preserves the edge-labeling. Thus, $\pre{v}{\sigma(\lambda)} = \id$ for every $v \in \mathcal{L}^{edge}_o$ and $\lambda = \Id_X$. We conclude that any $\lambda \in \Lambda$ that fixes a vertex is the identity.

	
	

	Now suppose $\lambda$ fixes some point $p \in X$ that may not be a vertex. Let $C$ be the lowest-dimensional cube that contains $p$. Since $\lambda$ fixes $p$, it has to send $C$ to itself while preserving the labels of the oriented edges of $C$. It is an easy exercise that every vertex $q$ in $C$ is uniquely determined by the set of labels carried by its outgoing edges in $C$ (because these labels indicate on which side of every hyperplane in $C$ the vertex $q$ lies). Therefore, since $\lambda$ preserves the labeling, it has to fix the vertices in $C$. The previous argument thus applies and $\lambda = \Id_X$.
\end{proof}

\begin{lemma} \label{lem:labeling.implies.special}
	If $X$ admits a $\Lambda$-invariant cubical edge-labeling, then $S$ is special.
\end{lemma}

\begin{proof}
	Let $\ell$ be a $\Lambda$-invariant cubical edge-labeling on $X$. We have to check that the four pathologies that are excluded in special cube complexes do not occur in $S$. First, we note that, since $\ell$ is $\Lambda$-invariant, the edge-labeling on $X$ descends to an edge-labeling $\ell_S : \overrightarrow{\mathcal{E}}(S) \rightarrow \Sigma$ on $S$. Note that, if $e, e' \in \overrightarrow{\mathcal{E}}(S)$ are in the same oriented parallel class, then $\ell_S(e) = \ell_S(e')$ since property (2) of cubical edge-labelings descends from $X$ to $S$.
		
	Suppose there exists a hyperplane $\hat{h}$ in $S$ that self-intersects. Then there exist two oriented edges $e, e'$ in $S$ that are parallel and span a square. Since $e$ and $e'$ are parallel, we have that $\ell_S(e) = \ell_S(e')$. The square spanned by $e$ and $e'$ lifts to a square in $X$ spanned by two edges $f$ and $f'$ that satisfy $\ell(f) = \ell_S(e) = \ell_S(e') = \ell(f')$. This contradicts property (1) of a cubical edge-labeling. We thus conclude that $\hat{h}$ cannot exist.
	
	Suppose there exists a hyperplane $\hat{h}$ in $S$ that is not two-sided. Then there exists an oriented edge $e$ in $S$ such that $\ell_S(e) = \ell_S(e^{-1})$. We lift $e$ to an edge $f$ in $X$ which then satisfies $\ell(f) = \ell(f^{-1})$. However, by property (4) of cubical edge-labelings, $\ell(f^{-1}) = \ell(f)^{-1}$ and $\cdot^{-1}$ is fixed-point free. Thus, we arrive at a contradiction and every hyperplane in $S$ has to be two-sided.
	
	Suppose there exists a self-osculating hyperplane $\hat{h}$ in $S$. Then there exists a vertex in $S$ with two outgoing edges $e, e'$ that are parallel and thus satisfy $\ell_S(e) = \ell_S(e')$. This lifts to a vertex in $X$ with two outgoing edges that have the same label, contradicting property (1).
	
	Suppose there exist two hyperplanes $\hat{h}, \hat{k}$ in $S$ that inter-osculate. Then there exist two vertices $p, p'$ in $S$ with outgoing edges $e_1, e_2 \in \overrightarrow{\mathcal{E}}(S)_{p}$ and $e'_1, e'_2 \in \overrightarrow{\mathcal{E}}(S)_{p'}$ such that $e_1$ and $e_2$ span a square, $e'_1$ and $e'_2$ do not span a square, $\ell_S(e_1) = \ell_S(e'_1)$, and $\ell_S(e_2) = \ell_S(e'_2)$. This lifts to two vertices $q, q'$ in $X$ and edges $f_1, f_2 \in \overrightarrow{\mathcal{E}}(X)_{q}$ and $f'_1, f'_2 \in \overrightarrow{\mathcal{E}}(X)_{q'}$ such that $f_1$ and $f_2$ span a square, $f'_1$ and $f'_2$ do not span a square, $\ell(f_1) = \ell(f'_1)$, and $\ell(f_2) = \ell(f'_2)$. This contradicts property (3).
	
	We conclude that all four pathologies do not occur in $S$. Therefore $S$ is special.
\end{proof}

The three lemmas above together prove Theorem \ref{thm:characterising.existence.of.cubical.edge.labelings}. This result tells us that we can easily produce a $\CAT$ cube complex with an action $\Lambda \curvearrowright X$ and a $\Lambda$-invariant cubical edge-labeling by starting with a compact, special, non-positively curved cube complex $S$ and taking its universal covering. Furthermore, every cube complex to which our results apply can be obtained in this way. Theorem \ref{thm:characterising.existence.of.cubical.edge.labelings} implies in particular that, if $X$ admits a $\Lambda$-invariant cubical edge-labeling, then $\Lambda$ acts geometrically on a locally finite (and thus proper) $\CAT$ cube complex. By the Milnor-Schwartz-Lemma, this implies that $\Lambda$ is finitely generated.

\begin{remark} \label{rem:Salvetti.complexes}
	A key example that admits $\Lambda$-invariant cubical edge-labelings are Salvetti-complexes of right-angled Artin groups. These corresponds exactly to those pairs $(X, \Lambda)$ where $\Lambda$ acts vertex-transitively. The theory of $\Aut(X)$ becomes significantly simpler when we restrict to this subclass of $\CAT$ cube complexes with special, compact quotient, as can be seen for example in the form of the finite generating set described in Theorem \ref{thm:fin.top.gen.set}.
\end{remark}




\section{Topological generators for stabilizer subgroups} \label{sec:topological.generators.for.stabilizers}

Let $o \in X^{(0)}$ be a vertex. We denote the stabilizer of $o$ by $G_o < \Aut(X)$. In this section, we present a technique that produces a family of cubical isometries that generates a dense subgroup of $G_o$. We then use this family to show that $\Aut(X)$ is topologically finitely generated, that is, it has a finitely generated dense subgroup. This technique relies on an assumption about $X$. It is unclear whether this assumption is always true. However, we know it to be true for Salvetti-complexes of right-angled Artin groups (see \cite{HartnickMedici23b}). To formulate this assumption, we need some terminology.

\begin{definition} \label{def:vertex.types}
	Let $v, w \in \mathcal{L}_o^{red}$ such that $v \neq \epsilon$. Let $w'$ be a reduced word such that $vw'^{\circ} = w^{\circ}$. We say that $w$ is
	\begin{itemize}
		\item of {\it Type 1 relative to $(o,v)$} if $vw'$ is reducible.
		
		\item of {\it Type 2 relative to $(o,v)$} if $vw'$ is reduced and there exists $s \in \reduc(v)$ such that $[w', s] = 1$.
		
		\item of {\it Type 3 relative to $(o,v)$} if $vw'$ is reduced and for every $s \in \reduc(v)$, $[w',s] \neq 1$.
	\end{itemize}
\end{definition}

\begin{remark}	\label{rem:equivalent.definition.of.vertex.types}
	A more geometric definition of Types is the following: Let $\hat{h}_1, \dots, \hat{h}_n$ be the hyperplanes that cross edges incident to $v^{\circ}$ and that separate $o$ from $v^{\circ}$. Let $h_i$ be the halfspace bounded by $\hat{h}_i$ that contains $v^{\circ}$. Then $w$ is of Type 1 relative to $(o,v)$ if and only if $w^{\circ}$ is not contained in $\cap_{i=1}^n h_i$. It is of Type 2 if and only if $w^{\circ}$ lies in $\cap_{i=1}^n h_i$ and is contained in the carrier of one of the $\hat{h}_i$. It is of Type 3 if and only if $w^{\circ}$ lies in $\cap_{i=1}^n h_i$ and does not lie in the carrier of any $\hat{h}_i$.
	
	It follows from this equivalence that the Type of a word $w$ only depends on its endpoint. Thus, we may also speak of the Type of a vertex relative to $(o,v)$. Note that $o$ is of Type 1 and $v$ is of Type 2 relative to $(o,v)$ for every $v \in \mathcal{L}_o^{red} \setminus \{ \epsilon \}$.
\end{remark}

Throughout this section, we will work with the $\ell^1$-metric on $X$. This is the metric obtained by equipping every cube $[0,1]^m$ in $X$ with the $\ell^1$-metric and taking the resulting shortest-path-metric. The distance between two vertices in the $\ell^1$-metric is equal to the number of hyperplanes separating the vertices. Furthermore, for every $v \in \mathcal{L}^{red}_o$, we have that $d_{\ell^1}(o, v^{\circ}) = length(v)$.

\begin{definition} \label{def:permutation.appearing.at.v}
	Let $v \in \mathcal{L}_o^{red}$ and $\tau : \Sigma_v^{edge} \rightarrow \Sigma_v^{edge}$ a bijection. We say that $\tau$ {\it appears at $v$ in $G_o$} if there exists a cubical isometry $g$ that fixes an $\ell^1$-ball of radius $r = length(v)$ and that satisfies $\pre{v}{\sigma(g)} = \tau$.
\end{definition}

\begin{notation} \label{not:Avt}
Let $v \in \mathcal{L}_o^{red}$ and let $\tau : \Sigma_v^{edge} \rightarrow \Sigma_v^{edge}$ be a bijection. If $v \neq \epsilon$, then the expression $A_{v, \tau}$ always denotes a cubical isometry that fixes all vertices of Type 1 and 2 relative to $(o,v)$ and that satisfies $\pre{v}{\sigma(A_{v,\tau})} = \tau$. If $v = \epsilon$, then $A_{\epsilon, \tau}$ denotes any cubical isometry that fixes $o$ and satisfies $\pre{\epsilon}{\sigma(A_{\epsilon,\tau})} = \tau$.

The cubical isometry $A_{v, \tau}$ is not unique in general and may not exist. When we write $A_{v, \tau}$, this denotes a choice of such a cubical isometry and requires that such a choice exists.
\end{notation}

\begin{lemma} \label{lem:properties.of.tau.that.appears.at.v}
	If $A_{v, \tau}$ exists, then $\tau$ appears at $v$ in $G_o$. Furthermore, if $\tau$ appears at $v$ in $G_o$, then $\tau$ preserves commutation and fixes $\langle \reduc(v) \rangle$ pointwise.
\end{lemma}

\begin{proof}
	By the definition of Types, any $w \in \mathcal{L}_o^{red}$ with $length(w) \leq length(v)$ is either of Type 1 relative to $(o,v)$, or $v^{\circ} = w^{\circ}$ and $w$ is of Type 2. This implies that, whenever $A_{v, \tau}$ exists, it fixes an $\ell^1$-ball of radius $r = length(v)$ centered at $o$. Thus, $\tau$ appears at $v$ in $G_o$ whenever $A_{v, \tau}$ exists.

	Let $g$ be a cubical isometry that fixes a ball of radius $length(v)$ centered at $o$ such that $\pre{v}{\sigma(g)} = \tau$. Since $g$ is a cubical isometry, all its local actions preserve commutation by Theorem \ref{thm:characterization.of.cubical.portraits}. By definition, for every $s \in \langle \reduc(v) \rangle$, we have that $vs$ is of Type 1 or 2 relative to $(o,v)$. Since $g$ has to fix all vertices of Type 1 or 2 relative to $(o,v)$, we see that $g$ fixes $v^{\circ}$ and $vs^{\circ}$ for all $s \in \langle \reduc(v) \rangle$. Thus $\tau = \pre{v}{\sigma(g)}$ has to fix $\langle \reduc(v) \rangle$ pointwise.
\end{proof}

For every $v \in \mathcal{L}_o^{red}$ and every bijection $\tau : \Sigma^{edge}_v \rightarrow \Sigma^{edge}_v$ that preserves commutation and fixes $\langle \reduc(v) \rangle$ pointwise, choose one element $A_{v, \tau}$, provided it exists, and denote the collection of these choices by $\mathcal{A}$. (Note that we also choose $A_{\epsilon, \tau}$ for all $\tau$ for which there exists one.) For the remainder of this section, $A_{v, \tau}$ denotes this specific choice.

The following is a slightly stronger Version of {\cite[Theorem B]{HartnickMedici23b}}. (The difference lies in Definition \ref{def:permutation.appearing.at.v}.)

\begin{theorem}
\label{thm:choice.of.A.generates.stabilizer}
	Suppose that for every $v \in \mathcal{L}_o^{red}$ and every $\tau$ that appears at $v$ in $G_o$, there exists a cubical isometry $A_{v, \tau}$. Then the collection $\mathcal{A}$ generates a dense subgroup of $G_o$.
\end{theorem}

The assumption made in this theorem is satisfied for example when $\Lambda$ is a right-angled Artin group and $X$ is its Salvetti-complex (see {\cite[Theorem 5.31]{HartnickMedici23b}}).

\begin{proof}	
	Let $g \in G_o$. The strategy of the proof is to consider the closed $\ell^1$-balls $B_n(o)$ of radius $n$ centered at $o$ and to do induction over $n$. Specifically, if $g$ fixes $B_n(o)$ pointwise, we show that we can write it as a product of finitely many $A_{v, \tau} \in \mathcal{A}$, together with an element $g'$ that fixes $B_{n+1}(o)$ pointwise. By induction, this will imply that every $g \in G_o$ can be written as the limit of a sequence of elements in $\langle \mathcal{A} \rangle$.
	
	Suppose $g_{n} \in G_o$ fixes the ball $B_n(o)$ for some $n \geq 0$. For every vertex $p$ at distance $n$ from $o$, choose a reduced word $v \in \mathcal{L}_o^{red}$ such that $v^{\circ} = p$. Since $X$ is locally finite, this yields a finite collection of words $v_1, \dots, v_N$. For every $v_i$, define $\tau_i := \pre{v_i}{\sigma(g_n)}$. Since $g_n$ fixes $B_n(o)$ pointwise, $\tau_i$ appears at $v_i$ in $G_o$ for every $i$. By assumption, this implies that there exists a cubical isometry $A_{v_i, \tau_i} \in \mathcal{A}$. Define a cubical isometry
	\[ g_{n+1} := \left( \prod_{i=1}^N A_{v_i, \tau_i}^{-1}  \right) \circ g_n. \]
	
	As we will see, the restriction of $g_{n+1}$ to $B_{n+1}(o)$ does not depend on the ordering in the product, which is why we may choose any order. We will show that $g_{n+1}$ fixes $B_{n+1}(o)$ pointwise. Since every reduced word of length $\leq n$ is of Type 1 or 2 with respect to $(o,v_i)$ for every $i$, it is clear that the $A_{v_i, \tau_i}$ fix $B_n(o)$ pointwise. Therefore, $g_{n+1}$ fixes $B_n(o)$ pointwise. For the vertices at distance $n+1$ from $o$, we have to work more carefully. We require three steps.
	
	\subsubsection*{Step 1: Every vertex at distance $n+1$ from $o$ gets moved by at most one $A_{v_i, \tau_i}$.}
	
	Let $p$ be a vertex at distance $n+1$ from $o$. Let $j$ be an index such that there is no $t \in \Sigma^{edge}_{v_j}$ that satisfies $v_j t^{\circ} = p$. We start by showing that $p$ is fixed by $A_{v_j, \tau_j}$. Let $w' \in \mathcal{L}_{v_j}^{edge}$ be a reduced word such that $v_j w'^{\circ} = p$. We claim that the word $v_j w'$ is reducible. Indeed, if it was reduced, then its length would be equal to the distance from $o$ to $p$, which is $n+1$. However, if $v_j w'$ has length $n+1$, then $w'$ has length $1$ and the pair $(v_j, w')$ would be a pair $(v_j, t)$ such that $v_j t^{\circ}Ê= p$. But $j$ is an index for which this does not exist, a contradiction. Thus $v_j w'$ is reducible, which implies that $p$ is of Type 1 with respect to $(o,v_j)$. Therefore $A_{v_j, \tau_j}$ fixes $p$ whenever there exists no $t \in \Sigma^{edge}_{v_j}$ such that $v_j t^{\circ} = p$.
	
	Now let $i$ be an index such that there exists $s \in \Sigma^{edge}_{v_i}$ such that $v_i s^{\circ} = p$. If there exists only one such pair, then $A_{v_j, \tau_j}$ fixes $p$ for every $j \neq i$ by the argument above and we are done. Suppose thus that there are at least two pairs $(v_i, s_i)$, $(v_j, s_j)$ such that $v_i s_i^{\circ} = v_j s_j^{\circ}$. This implies that $s_i^{-1}, s_j^{-1} \in \reduc(v_i s_i)$ and thus $[s_i, s_j] = [s_i^{-1}, s_j^{-1}] = 1$. Let $v_0$ be a reduced word such that $v_0^{\circ} = \left( v_i s_i s_i^{-1} s_j^{-1} \right)^{\circ}$. This word is of length $n-1$, as it is reduced, and it satisfies $v_0 s_i s_j^{\circ} = p = v_i s_i^{\circ}$. By induction assumption, we know that $g_n$ fixes all edges incident to $v_0^{\circ}$, since these edges are all contained in $B_n(o)$. Therefore, $\pre{v_0}{\sigma(g_n)} = \id$. Using $par$ and $end$, we compute
	\begin{equation*}
		\begin{split}
			\tau_i(s_i) & = \pre{v_i}{\sigma(g_n)}(s_i)\\
			& = \pre{v_0 s_j}{\sigma(g_n)}(s_i)\\
			& = \pre{v_0}{\sigma(g_n)}(s_i)\\
			& = s_i. 
		\end{split}
	\end{equation*}
	By definition, $A_{v_i, \tau_i}(v_i s_i) = v_i \tau_i(s_i)$. The computation above thus shows that $A_{v_i, \tau_i}(v_i s_i) = v_i s_i$ and thus $A_{v_i, \tau_i}$ fixes the vertex $v_i s_i^{\circ} = p$. By an analogous argument, the same holds for $A_{v_j, \tau_j}$. We conclude that, whenever there are at least two pairs $(v_i, s_i)$, $(v_j, s_j)$ such that $v_i s_i^{\circ} = p = v_j s_j^{\circ}$, then all $A_{v_i, \tau_i}$ fix the vertex $p$. This proves Step 1.
	
	\subsubsection*{Step 2: Let $i \in \{ 1, \dots, N \}$ and let $s \in \Sigma^{edge}_{v_i}$. If $A_{v_i, \tau_i}(v_i s) \neq vs$, then 	for every $j \neq i$, $A_{v_j, \tau_j}$ fixes $A_{v_i, \tau_i}(v_i s)$.}
	
	Suppose $A_{v_i, \tau_i}(v_i s) \neq v_i s$. Since $A_{v_i, \tau_i}$ is bijective, it cannot send $v_i s$ and $A_{v_i, \tau_i}(v_i s)$ to the same word. Therefore, it cannot fix $A_{v_i, \tau_i}(v_i s)$. By Step 1, this implies that for every $j \neq i$, we have that $A_{v_j, \tau_j}$ fixes $A_{v_i, \tau_i}(v_i s) = v_i \tau_i(s)$.
	
	\subsubsection*{Step 3: Compute $g_{n+1}$ for reduced words of length $n+1$.}
	
	Let $w \in \mathcal{L}_o^{red}$ be a reduced word of length $n+1$. There exists $i \in \{ 1, \dots, N \}$ and $s \in \Sigma_{v_i}$ such that $w^{\circ} = v_i s^{\circ}$. If the pair $(v_i, s)$ is not unique, we conclude from Step 1 that $\tau_i(s) = s$ and thus we compute
	\[ g_{n}(v_i s) = v_i \tau_i(s) = v_i s, \]
	\[ A_{v_j, \tau_j}( v_i s ) = v_i s \]
	for every $j \in \{ 1, \dots, N \}$. We conclude that
	\[ g_{n+1}(v_i s) = \prod_{i=1}^N A_{v_j, \tau_j}^{-1} \circ g_n (v_i s) = v_i s \]
	and thus $g_{n+1}$ fixes $v_i s^{\circ}$.
	
	If the pair $(v_i, s)$ is unique, then $A_{v_j, \tau_j}$ fixes $v_i s$ and $A_{v_i, \tau_i}(v_i s)$ for every $j \neq i$. In particular $A_{v_j, \tau_i}^{-1}$ fixes these two words as well. Using that $g_n(v_i s) = v_i \tau_i(s) = A_{v_i, \tau_i}(v_i s)$, we compute that
	\begin{equation*}
		\begin{split}
			g_{n+1}(v_i s) & = \prod_{j=1}^N A_{v_j, \tau_j}^{-1} \circ g_n(v_i s)\\
			& = \prod_{j=1}^N A_{v_j, \tau_j}^{-1} (v_i \tau_i(s))\\
			& = A_{v_i, \tau_i}^{-1}( v_i \tau_i(s) )\\
			& = v_i s.
		\end{split}
	\end{equation*}
	We conclude that $g_{n+1}$ fixes $v_i s^{\circ}$ in this case as well. This implies that $g_{n+1}$ fixes the closed $\ell^1$-ball of radius $n+1$ around $o$. This finishes Step 3.\\
	
	Using induction over $n$, we see that for any $n \geq 0$ and any $g \in G_o$, we find an element $g_n \in G_o$ such that $g_n = \prod_{i=1}^{M_n} A_{v_i, \tau_i}^{-1} \circ g$ and $g_n$ fixes the $\ell^1$-ball of radius $n$ around $o$. We conclude that the sequence $(g_n)_n$ converges to $\Id_X$ in compact-open topology and thus the sequence of products $\left( \prod_{i=M_n}^{1} A_{v_i, \tau_i} \right)_n$ converges to $g$ in compact-open topology. Therefore, the group generated by the set $\mathcal{A} = \{ A_{v, \tau} \}$ is dense in $G_o$.
		
\end{proof}

We can use this topologically generating set of $G_o$ to produce a topologically generating set of $\Aut(X)$. To do so, we need some preliminary results. Let $g \in \Lambda$ and $u \in \mathcal{L}^{edge}_o$ a word such that $u^{\circ} = g(o)$. We define a map
\[ L_u : \mathcal{L}^{edge}_u \rightarrow \Sigma^* \]
\[ v \mapsto uv. \]

\begin{lemma} \label{lem:Lambda.as.left.multiplication}
	The map $L_u$ defines an automorphism on $X$. This automorphism coincides with $g$.	
\end{lemma}

\begin{proof}
	We show that, for every $v \in \mathcal{L}^{edge}_o$, $uv^{\circ} = g(v^{\circ})$. (In particular, $uv \in \mathcal{L}^{edge}_o$.) This proves both statements of the Lemma.
	
	By construction, $L_u(\epsilon) = u$ and by definition, $u^{\circ} = g(o)$. Let $\gamma_0$ be the edge-path that starts at $o$ and corresponds to $u$. Its endpoint is $g(o)$. Let $v \in \mathcal{L}^{edge}_o$ and let $\gamma$ be the corresponding edge-path in $X$ that starts at $o$. The endpoint of $\gamma$ is $v^{\circ}$. Since $g \in \Lambda$, it preserves the edge-labeling and $g(\gamma)$ is an edge-path that starts at $g(o)$ and spells the same word as $\gamma$. Thus, the edge-paths $\gamma_0$ and $g(\gamma)$ can be concatenated to a path from $o$ to $g(v^{\circ})$ that spells the word $uv$. We conclude that $uv \in \mathcal{L}^{edge}_o$ and that $uv^{\circ} = g(v^{\circ})$. This implies that $L_u$ induces a map on vertices that coincides with $g$, which is the desired statement.
\end{proof}

We denote the isomorphism induced by $L_u$ by $L_u$ as well. We call the $L_u$ `left-multiplication maps'. We now look at conjugations of $A_{v, \tau}$ by elements in $\Lambda$. By the lemma above, we can treat this as conjugation by left-multiplication maps.

\begin{lemma} \label{lem:conjugations.of.Avt}
	Let $v, v' \in \mathcal{L}^{red}_o$ such that $\reduc(v) = \reduc(v')$ and there exists $g \in \Lambda$ such that $g(v^{\circ}) = v'^{\circ}$. Let $\tau : \Sigma^{edge}_v \rightarrow \Sigma^{edge}_v$ be a bijection that appears at $v$ in $G_o$. Then $\tau$ appears at $v'$ in $G_o$ and $g \cdot A_{v, \tau} \cdot g^{-1}$ is an $A_{v', \tau}$.
\end{lemma}

\begin{proof}
	If we can show that $g \cdot A_{v, \tau} \cdot g^{-1}$ is an $A_{v', \tau}$, then $\tau$ appears at $v'$ in $G_o$.  Let $u \in \mathcal{L}^{red}_o$ be such that $g = L_u$ and denote $A := g \cdot A_{v, \tau} \cdot g^{-1}$. To show that $A$ is an $A_{v', \tau}$, we start by showing that it fixes all words of Type 1 and 2 relative to $(o,v')$.
	
	\subsubsection*{Step 1: Show that for any $w \in \mathcal{L}^{edge}_o$ that $w$ is of the same Type relative to $(o,v)$ as $uw$ relative to $(o,v')$.}
	
	Let $w'$ be a reduced word such that $vw'^{\circ}Ê= w^{\circ}$. Then $uw^{\circ} = uvw'^{\circ} = v'w'^{\circ}$. Since $\reduc(v) = \reduc(v')$, the word $vw'$ is reduced if and only if $v'w'$ is reduced. We immediately conclude that $w$ is of Type 1 relative to $(o,v)$ if and only if $uw$ is of Type 1 relative to $(o,v')$.
	
	Suppose now that $vw'$, and thus $v'w'$, is reduced. The word $w$ is of Type 2 relative to $(o,v)$ if and only if there exists $s \in \reduc(v)$ such that $[s,w'] = 1$. Since $\reduc(v) = \reduc(v')$, this happens if and only if there exists $s \in \reduc(v')$ such that $[s,w'] = 1$. We conclude that $w$ is of Type 2 relative to $(o,v)$ if and only if $uw$ is of Type 2 relative to $(o,v')$. Equivalence for Type 3 now follows since all other cases have been dealt with.
	
	\subsubsection*{Step 2: Show that $L_u \cdot A_{v, \tau} \cdot L_u^{-1}$ fixes all words of Type 1 and 2 relative to $(o,v')$.} Let $w$ be of Type 1 or 2 relative to $(o,v')$. Applying Step 1 with the word $u^{-1}$, we conclude that $u^{-1} w$ is of Type 1 or 2 relative to $(o,v)$. Therefore, $A_{v, \tau}$ fixes $u^{-1} w$ and we conclude that
	\[ L_u \cdot A_{v, \tau} \cdot L_u^{-1} (w)^{\circ} = L_u \cdot A_{v, \tau} (u^{-1} w)^{\circ} = L_u( u^{-1} w)^{\circ} = u u^{-1} w^{\circ} = w^{\circ}. \]
	Thus $L_u \cdot A_{v, \tau} \cdot L_u^{-1}$ fixes all vertices of Type 1 and 2 relative to $(o,v')$. (Recall that the Type of a word depends only on its endpoint.) This finishes Step 2.\\
	
	We are left to show that $\pre{ v' }{\sigma(A)} = \tau$. For any $s \in \Sigma^{edge}_{v'}$, we compute
	\begin{equation*}
		\begin{split}
			A(v's) & = L_u \cdot A_{v, \tau} ( u^{-1} v' s)\\
			& = L_u \left( A_{v,\tau}(u^{-1} v') \pre{u^{-1}v'}{ \sigma(A_{v, \tau})}(s) \right)\\
			& = L_u \left( u^{-1} v' \tau(s) \right)\\
			& = u u^{-1} v' \tau(s).
		\end{split}
	\end{equation*}
	The induced cubical isometry thus sends the outgoing edge at $v'^{\circ}$ labeled by $s$ to the outgoing edge at $u u^{-1} v'^{\circ} = v'^{\circ}$ labeled by $\tau(s)$. We conclude that $\pre{ v' }{\sigma(A)} = \tau$. Therefore, $A$ satisfies the conditions of an $A_{v', \tau}$, which finishes the proof.
\end{proof}

We now build a topologically generating set of $\Aut(X)$. Let $\mathcal{A}$ denote the same set as in Theorem \ref{thm:choice.of.A.generates.stabilizer}. Let $S$ be a finite generating set of $\Lambda$ such that $S^{-1} = S$. (Since $\Lambda$ acts geometrically on $X$, it is finitely generated.) Choose a vertex $o \in X^{(0)}$. The orbit $\Aut(X) \cdot o$ can be written as a finite union of $\Lambda$-orbits. Let $g_0, \dots, g_n \in \Aut(X)$ such that $g_0 = \Id_X$ and
\[ \coprod_{i=0}^n \Lambda \cdot g_i(o) = \Aut(X) \cdot o. \]
Finally, consider the equivalence relation on $\mathcal{L}^{red}_o$ under which two words $v, w \in \mathcal{L}^{red}_o$ are equivalent if and only if $\reduc(v) = \reduc(w)$ and there exists $g \in \Lambda$ such that $g(v^{\circ}) = w^{\circ}$. Since $\Lambda$ acts cocompactly and the sets $\reduc(v)$ are subsets of the finite set $\Sigma$, there are only finitely many equivalence classes. For each equivalence class, choose a representative $v_i$. This yields a finite sequence $v_0, \dots, v_m \in \mathcal{L}^{red}_o$. (Without loss of generality, $v_0 = \epsilon$.) We now consider the set
\[ \mathcal{A}_0 := S \cup \{ g_0, \dots, g_n \} \cup \{ A_{v_i, \tau} \vert 0 \leq i \leq m, \tau \text{ appears at $v_i$ in $G_o$ } \}. \]
We emphasize that this set is finite.

\begin{theorem} \label{thm:fin.top.gen.set}
	Suppose that for every $0 \leq i \leq m$ and every $\tau$ that appears at $v_i$ in $G_o$, there exists a cubical isometry $A_{v_i, \tau}$. Then the subgroup generated by $\mathcal{A}_0$ is dense in $\Aut(X)$. In particular, $\Aut(X)$ is generated by $\mathcal{A}_0 \cup G_o$ and is thus compactly generated.
\end{theorem}

For the same reason as in Theorem \ref{thm:choice.of.A.generates.stabilizer}, the assumption of this theorem is satisfied if $\Lambda$ is a right-angled Artin group and $X$ its Salvetti-complex. Also note that we have reduced the ``infinitely large'' condition of Theorem \ref{thm:choice.of.A.generates.stabilizer} to a ``finite'' condition concerning finitely many $v_i$ and finitely many $\tau$. This opens up new possibilities to check this assumption.

\begin{proof}
	Let $g \in \Aut(X)$. Since $\coprod_{i=0}^n \Lambda \cdot g_i(o) = \Aut(X) \cdot o$, there exists $0 \leq i \leq n$ such that $g(o) \in \Lambda \cdot g_i(o)$. Thus, there exists $\lambda \in \Lambda$ such that $g_i^{-1} \cdot \lambda^{-1} \cdot g(o) = o$. We conclude that $S \cup \{ g_0, \dots, g_n \} \cup G_o$ generates $\Aut(X)$ which is thus compactly generated.
	
	By Theorem \ref{thm:choice.of.A.generates.stabilizer}, this implies that the union $S \cup \{ g_0, \dots, g_n \} \cup \mathcal{A}$ generates a dense subgroup of $\Aut(X)$. We are left to reduce from $\mathcal{A}$ to $\mathcal{A}_0$. Let $A_{v, \tau} \in \mathcal{A}$. Then there exists $0 \leq i \leq m$ such that $v_i$ is equivalent to $v$, that is, $\reduc(v_i) = \reduc(v)$ and there exists $\lambda \in \Lambda$ such that $\lambda(v_i^{\circ}) = v^{\circ}$. By Lemma \ref{lem:Lambda.as.left.multiplication}, there exists $u \in \mathcal{L}^{red}_o$ such that $L_u = \lambda$. By Lemma \ref{lem:conjugations.of.Avt}, we conclude that $A_{v, \tau} = L_u \cdot A_{v_i, \tau} \cdot L_u^{-1}$. Since $\Lambda$ is generated by $S$, we conclude that every element of $\mathcal{A}$ is contained in the group generated by $S \cup \{ g_0, \dots, g_n \} \cup \{ A_{v_i, \tau} \vert 0 \leq i \leq m, \tau \text{ appears at $v_i$ in $G_o$ } \}$. Therefore, $\mathcal{A}$ is contained in the group generated by $\mathcal{A}_0$, which completes the proof.
\end{proof}

\begin{remark}
	The collection $g_1, \dots g_n$ poses a challenge in producing an explicit description of this finite set. This can be circumvented if the action of $\Lambda$ on $X$ is vertex-transitive. In that case, $\Lambda \cdot o = X^{(0)}$ and thus the $\Aut(X)$-orbit has to coincide with the $\Lambda$-orbit, making the elements $g_1, \dots, g_n$ unnecessary. As was mentioned in Remark \ref{rem:Salvetti.complexes}, this is equivalent to requiring $\Lambda$ to be a right-angled Artin group with $X$ being its Salvetti-complex.
\end{remark}




\appendix

\section{Proof of Theorem \ref{thm:characterization.of.cubical.portraits} and Theorem \ref{thm:canonical.portrait.extension}} \label{sec:proofs}

The proof of Theorem \ref{thm:characterization.of.cubical.portraits} can be split into several parts. First, we show that, if $g : X \rightarrow Y$ is a cubical isometry, then the family $(\pre{v}{\sigma(g)})_{v \in \mathcal{L}_o^{edge}}$ defines a cubical portrait.

\begin{lemma} \label{lem:first.direction.of.main.theorem}
	Let $g : X \rightarrow Y$ be a cubical isometry. Then the family $(\pre{v}{\sigma(g)})_{v \in \mathcal{L}_o^{edge}}$ defines a cubical portrait.
\end{lemma}

\begin{proof}
	Choose some vertex $o \in X^{(0)}$ and let $o' := g(o)$. It is clear that $g$ defines a rooted tree-isomorphism $T_o^{edge} \rightarrow T'^{edge}_{o'}$ and thus $(\pre{v}{\sigma(g)})_{v \in \mathcal{L}_o^{edge}}$ is a portrait on $T_o^{edge}$ to $\epsilon$ in $T'^{edge}_{o'}$. We are left to show that this portrait is cubical.
	
	Condition $tree$: As established in Theorem \ref{thm:characterising.tree.isomorphisms}, for any rooted tree-isomorphism $g : T_o^{edge} \rightarrow T'^{edge}_{o'}$ and any $v \in \mathcal{L}^{edge}_o$, the map $\pre{v}{\sigma(g)}$ is a bijection from $\Sigma^{edge}_v$ to $\Sigma'^{edge}_{g(v)}$.
	
	Condition $comm$: Let $v \in \mathcal{L}_o^{edge}$ and $s, t \in \Sigma^{edge}_v$. By construction, $\pre{v}{\sigma(g)}$ describes the local action of the cubical isometry $g$ on the outgoing edges at $v^{\circ}$. In particular, the outgoing edges labeled $s$ and $t$ span a square if and only if the outgoing edges at $g(v)$ labeled $\pre{v}{\sigma(g)}(s)$ and $\pre{v}{\sigma(g)}(t)$ span a square. Thus, $[s,t] = [\pre{v}{\sigma(g)}(s), \pre{v}{\sigma(g)}(t)]$.
	
	Condition $par$: Let $v \in \mathcal{L}_o^{edge}$ and let $s, t \in \Sigma_v^{edge}$ such that $[s,t] = 1$. Let $e$ denote the oriented edge that starts at $v^{\circ}$ and is labeled $t$. Consider the square at $v^{\circ}$ spanned by the outgoing edges labeled $s$ and $t$. This square contains an oriented edge, denote it $e'$, that starts at $vs^{\circ}$ and is parallel to $e$. Since $g$ is a cubical isometry, it sends $g(e)$ and $g(e')$ into the same oriented parallel class. Thus, $\ell'(g(e)) = \ell'(g(e'))$ and we calculate
\[ \pre{v}{\sigma(g)}(t) = \ell'( g(e) ) = \ell'( g(e') ) = \pre{vs}{\sigma(g)}(t). \]

	Condition $inv$: Let $v \in \mathcal{L}_o^{edge}$ and $s \in \Sigma_v^{edge}$. Let $e$ be the oriented edge that starts at $v^{\circ}$ and is labeled $s$. Recall that $e^{-1}$ denotes the same edge with reverse orientation. By definition of cubical edge-labelings, $e^{-1}$ is labeled by $s^{-1}$. Since $g$ is a cubical isometry, $g(e)$ and $g(e')$ are the same edge with opposite orientations. Thus,
\[  \pre{v}{\sigma(g)}(s)^{-1} = \ell'( g(e) )^{-1} = \ell'( g(e)^{-1} ) = \ell'( g(e^{-1}) ) = \pre{vs}{\sigma(g)}(s^{-1}). \]

	Condition $end$: Let $v, w \in \mathcal{L}_o^{edge}$ such that $v^{\circ} = w^{\circ}$. Let $s \in \Sigma_v^{edge}$ and $e$ the oriented edge starting at $v^{\circ}$ labeled $s$. By definition of the local actions, we compute
\[ \pre{v}{\sigma(g)}(s) = \ell'( g(e) ) = \pre{w}{\sigma(g)}(s). \]

	Thus, $(\pre{v}{\sigma(g)})_{v \in \mathcal{L}_o^{edge}}$ is a cubical portrait.	
\end{proof}

Before we prove the converse, we need some preparation. Let $(\pre{v}{\sigma})_{v \in \mathcal{L}_o^{edge}}$ be a cubical portrait and define $\sigma$ according to equation (\ref{eq:definition.of.sigma}). If $\sigma$ is to induce a cubical isometry, the following is the first important thing that needs to be shown.

\begin{proposition} \label{prop:cubical.isometries.are.well.defined.on.vertices}
	Let $(\pre{v}{\sigma})_{v \in \mathcal{L}_o^{edge}}$ be a cubical portrait and let $v, w \in \mathcal{L}_o^{edge}$. Then $v^{\circ} = w^{\circ}$ if and only if $\sigma(v)^{\circ} = \sigma(w)^{\circ}$.
	
\end{proposition}

Suppose $v, w \in \mathcal{L}_o^{edge}$ such that $v^{\circ} = w^{\circ}$. Identifying $v$ and $w$ with the edge-paths starting at $o$ that they represent, we see that these are two edge-paths that have the same starting point and the same endpoint. The proof of this proposition relies on the idea that any two such edge-paths can be transformed into each other by using a finite sequence of two moves. This is the content of the following Lemma, which is a standard result about $\CAT$ cube complexes.

\begin{lemma} \label{lem:moves.for.edge.paths}
	Let $v, w \in \mathcal{L}_o^{edge}$ such that $v^{\circ} = w^{\circ}$. Then $v$ can be transformed into $w$ by finitely many applications of the following two moves:
\[ u_1 u_2 \leftrightarrow u_1 s s^{-1} u_2, \]
\[ u_1 s t u_2 \leftrightarrow u_1 t s u_2, \]
where $[s,t] = 1$ in the second move.
\end{lemma}

We refer to the first move as the cancellation-move and to the second as the commutation-move. The following Lemma is the key to the proof of Proposition \ref{prop:cubical.isometries.are.well.defined.on.vertices}. Note that we do not need property $tree$ for this Lemma.

\begin{lemma} \label{lem:moves.are.carried.over}
	Let $v, w \in \mathcal{L}^{edge}_o$ and let $(\pre{v}{\sigma})_{v \in \mathcal{L}^{edge}_o}$ be a portrait that satisfies $comm$, $par$, $inv$, and $end$. Then $v$ and $w$ are related by a finite sequence of the moves in Lemma \ref{lem:moves.for.edge.paths} if and only if $\sigma(v)$ and $\sigma(w) \in \Sigma'^{*}$ are.
	
	In fact, applying one cancellation-move to $v$ corresponds to applying one cancellation-move to $\sigma(v)$ and applying one commutation-move to $v$ corresponds to applying one commutation-move to $\sigma(v)$.
\end{lemma}

\begin{proof}
	Let $u_1 \in \mathcal{L}_o^{edge}$ and $u_2 = t_1 \dots t_l \in \mathcal{L}_{u_1}^{edge}$. In analogy to the notation $\pre{u_1}{\sigma(g)}(u_2)$, we define
	\[ \pre{u_1}{\sigma}(u_2) := \pre{u_1}{\sigma}(t_1) \dots \pre{u_1 t_1 \dots t_{i-1}}{\sigma}(t_i) \dots \pre{u_1 t_1 \dots t_{l-1}}{\sigma}(t_l). \]
	We first note that, because $(\pre{v}{\sigma})_{v \in \mathcal{L}_o^{edge}}$ satisfies $end$, we know that, if $u_1^{\circ} = u'^{\circ}_1$, then $\pre{u_1}{\sigma}(u_2) = \pre{u'_1}{\sigma}(u_2)$. Let $v, w \in \mathcal{L}_o^{edge}$. For both directions of the statement, we only have to consider the case that $v$ and $w$ (or $\sigma(v)$ and $\sigma(w)$) are related by one cancellation-move or one commutation-move
	
	$``\Rightarrow''$: Suppose there are $u_1, s, u_2$ such that
	\[ v = u_1 s s^{-1} u_2, \quad w = u_1 u_2. \]
	As noted above, condition $end$ implies that
	\[ \pre{u_1}{\sigma}(u_2) = \pre{u_1 s s^{-1}}{\sigma}( u_2 ). \]
	Condition $inv$ implies that
	\[ \pre{u_1}{\sigma}(s)^{-1} = \pre{u_1 s}{\sigma}(s^{-1}).  \]
	By equations (\ref{eq:main.formula}) and (\ref{eq:main.formula.extended}), we know that
	\begin{equation*}
		\begin{split}
			\sigma(v) & = \sigma(u_1 s s^{-1} u_2)\\
			& = \sigma(u_1) \pre{u_1}{\sigma}(s) \pre{u_1 s}{\sigma}(s^{-1}) \pre{u_1 s s^{-1}}{\sigma}(u_2)\\
			& = \sigma(u_1) \pre{u_1}{\sigma}(s) \pre{u_1}{\sigma}(s)^{-1} \pre{u_1}{\sigma}(u_2)
		\end{split}
	\end{equation*}
	and
	\begin{equation*}
		\sigma(u_1) \pre{u_1}{\sigma}(u_2) = \sigma(w).
	\end{equation*}
	Thus, $\sigma(v)$ and $\sigma(w)$ are related by one cancellation-move.\\
	
	Now suppose there are $u_1, s, t, u_2$ such that $[s,t] = 1$ and
	\[ v = u_1 s t u_2, \quad w = u_1 t s u_2. \]
	By $end$, we know that
	\[ \pre{u_1 s t}{\sigma}(u_2) = \pre{u_1 t s}{\sigma}(u_2). \]
	By $comm$, we know that
	\[ [\pre{u_1}{\sigma}(s), \pre{u_1}{\sigma}(t)] = 1. \]
	Combining this with $par$, we obtain
	\[ \pre{u_1}{\sigma}(s) = \pre{u_1 t}{\sigma}(s), \]
	\[ \pre{u_1}{\sigma}(t) = \pre{u_1 s}{\sigma}(t). \]
	Following the same principle as in the previous case, we make the following two computations:
	\begin{equation*}
		\begin{split}
			\sigma(v) & = \sigma(u_1 s t u_2)\\
			& = \sigma(u_1) \pre{u_1}{\sigma}(s) \pre{u_1 s}{\sigma}(t) \pre{u_1 s t}{\sigma}(u_2)\\
			& = \sigma(u_1) \pre{u_1}{\sigma}(s) \pre{u_1}{\sigma}(t) \pre{u_1st}{\sigma}(u_2),
		\end{split}
	\end{equation*}
	\begin{equation*}
		\begin{split}
			\sigma(w) & = \sigma(u_1 t s u_2)^{\circ}\\
			& = \sigma(u_1) \pre{u_1}{\sigma}(t) \pre{u_1 t}{\sigma}(s) \pre{u_1 t s}{\sigma}(u_2)\\
			& = \sigma(u_1) \pre{u_1}{\sigma}(t) \pre{u_1}{\sigma}(s) \pre{u_1 s t}{\sigma}(u_2).
		\end{split}
	\end{equation*}
	We conclude that $\sigma(v)$ and $\sigma(w)$ are related by one commutation-move.\\

	$``\Leftarrow''$: Suppose there are $u'_1, s', u'_2$ such that
	\[ \sigma(v) = u'_1 s' s'^{-1} u'_2, \quad \sigma(w) = u'_1 u'_2. \]
	By construction of the map $\sigma : \mathcal{L}_o^{edge} \rightarrow \Sigma'^{*}$, we find the following elements:
	\begin{itemize}
		\item A word $u_1 \in \mathcal{L}_o^{edge}$ such that $\sigma(u_1) = u'_1$.
		
		\item A letter $s \in \Sigma^{edge}_{u_1}$ such that $\pre{u_1}{\sigma}(s) = s'$.
		
		\item A word $u_2 \in \mathcal{L}_{u_1}^{edge}$ such that $\pre{u_1}{\sigma}(u_2) = u'_2$.
	\end{itemize}
	Using $end$, $inv$, and equations (\ref{eq:main.formula}) and (\ref{eq:main.formula.extended}), we compute
	\begin{equation*}
		\begin{split}
			\sigma(u_1 s s^{-1} u_2) & = \sigma(u_1) \pre{u_1}{\sigma}(s) \pre{u_1 s}{\sigma}(s^{-1}) \pre{u_1 s s^{-1}}{\sigma}(u_2)\\
			& = u'_1 s' s'^{-1} u'_2\\
			& = \sigma(v).
		\end{split}
	\end{equation*}
	On the other hand,
	\[ \sigma(u_1 u_2) = \sigma(u_1) \pre{u_1}{\sigma}(u_2) = u'_1 u'_2 = \sigma(w). \]
	Since $\sigma$ is injective, we conclude that $v = u_1 s s^{-1} u_2$, while $w = u_1 u_2$. Thus, $v$ and $w$ are related by one cancellation-move.\\
		
	Suppose there are $u'_1, s', t', u'_2$ such that $[s',t'] = 1$ and
	\[Ê\sigma(v) = u'_1 s' t' u'_2, \quad \sigma(w) = u'_1 t' s' u'_2. \]
	As before, using the definition of $\sigma : \mathcal{L}_o^{edge} \rightarrow \Sigma^{*}$, we find the following elements:
	\begin{itemize}
		\item A word $u_1 \in \mathcal{L}_o^{edge}$ such that $\sigma(u_1) = u'_1$.
		
		\item A letter $s \in \Sigma^{edge}_{u_1}$ such that $\pre{u_1}{\sigma}(s) = s'$.
		
		\item A letter $t \in \Sigma^{edge}_{u_1}$ such that $\pre{u_1 }{\sigma}(t) = t'$. (By $par$ we know that $\sigma_{u_1 s}(t) = \sigma_{u_1}(t) = t'$.)
		
		\item A word $u_2 \in \mathcal{L}_{u_1 s t}^{edge}$ such that $\pre{u_1 s t}{\sigma}(u_2) = u'_2$.
	\end{itemize}
	Since $\pre{u_1}{\sigma}$ satisfies $comm$, we have that
	\[ [s,t] = [\pre{u_1}{\sigma}(s), \pre{u_1}{\sigma}(t)] = [s',t'] = 1.\]
	Using $par$ and equations (\ref{eq:main.formula}) and (\ref{eq:main.formula.extended}), we obtain
	\begin{equation*}
		\begin{split}
			\sigma(u_1 s t u_2) & = \sigma(u_1) \pre{u_1}{\sigma}(s) \pre{u_1 s}{\sigma}(t) \pre{u_1 s t}{\sigma}(u_2) = u'_1 s' t' u'_2 = \sigma(v)
		\end{split}
	\end{equation*}
	On the other hand, using $par$, $end$ and equations (\ref{eq:main.formula}) and (\ref{eq:main.formula.extended}), we compute
	\begin{equation*}
		\begin{split}
			\sigma(u_1 t s u_2) & = \sigma(u_1) \pre{u_1}{\sigma}(t) \pre{u_1 t}{\sigma}(s) \pre{u_1 t s}{\sigma}(u_2)\\
			& = u'_1 t' \pre{u_1}{\sigma}(s) \pre{u_1 s t}{\sigma}(u_2)\\
			& = u'_1 t' s' u'_2\\
			& = \sigma(w).
		\end{split}
	\end{equation*}
	Since $\sigma$ is injective, we conclude that $v = u_1 s t u_2$ and $w = u_1 t s u_2$. Thus, $v$ and $w$ are related by one commutation-move, which finishes the proof.
\end{proof}

\begin{proof}[Proof of Proposition \ref{prop:cubical.isometries.are.well.defined.on.vertices}]
	Let $v, w \in \mathcal{L}^{edge}_o$. Since $(\pre{v}{\sigma})_{v \in \mathcal{L}^{edge}_o}$ is a cubical portrait, we conclude from property $tree$ that $\sigma(v), \sigma(w) \in \mathcal{L}'^{edge}_{o'}$ and thus $\sigma(v)^{\circ}$ and $\sigma(w)^{\circ}$ are well-defined.
	
	Combining Lemma \ref{lem:moves.for.edge.paths} and Lemma \ref{lem:moves.are.carried.over}, we obtain that $v^{\circ} = w^{\circ}$ if and only if $v$ and $w$ are related by a finite sequence of cancellation-moves and commutation-moves, which holds if and only if the same holds for $\sigma(v)$ and $\sigma(w)$, which holds if and only if $\sigma(v)^{\circ} =Ê\sigma(w)^{\circ}$. This proves the Proposition.
\end{proof}

With Proposition \ref{prop:cubical.isometries.are.well.defined.on.vertices}, the last part of the proof of Theorem \ref{thm:characterization.of.cubical.portraits} becomes straight-forward.

\begin{lemma} \label{lem:converse.of.main.theorem}
	If $(\pre{v}{\sigma})_{v \in \mathcal{L}_o^{edge}}$ is a cubical portrait, then $\sigma$ induces a cubical isometry that sends $o$ to $o'$.
\end{lemma}

\begin{proof}
	By Proposition \ref{prop:cubical.isometries.are.well.defined.on.vertices}, $\sigma$ induces an injective map on vertices $g : X^{(0)} \rightarrow Y^{(0)}$. This map is bijective, since condition $tree$ implies that $\sigma$ is an isomorphism $T_o^{edge} \rightarrow T'^{edge}_{o'}$ on trees. We are left to show that two vertices $p, p' \in X^{(0)}$ are connected by an edge if and only if $g(p)$ and $g(p')$ are.
	
	Let $v, w \in \mathcal{L}_o^{edge}$ such that $v^{\circ}$ and $w^{\circ}$ are connected by an edge. Then there exists some $s \in \Sigma^{edge}_{v}$ such that $w^{\circ} = vs^{\circ}$. Therefore, we have
	\[ \sigma(w)^{\circ} = \sigma(vs)^{\circ} = \sigma(v) \pre{v}{\sigma}(s). \]
	This implies that $\sigma(v)^{\circ}$ and $\sigma(w)^{\circ}$ are connected by an edge (which is labeled by $\pre{v}{\sigma}(s)^{\pm 1}$, depending on orientation).
	
	Conversely, suppose $\sigma(v)$ and $\sigma(w)$ are connected by an edge. Then there exists $s' \in \Sigma'^{edge}_{\sigma(v)}$ such that $\sigma(w)^{\circ} = \sigma(v) s'^{\circ}$. Since $\pre{v}{\sigma} : \Sigma^{edge}_v \rightarrow \Sigma'^{edge}_{\sigma(v)}$ is bijective, there exists $s \in \Sigma^{edge}_v$ such that $\pre{v}{\sigma}(s) = s'$. We compute
	\[ \sigma(vs) = \sigma(v) \pre{v}{\sigma}(s) = \sigma(v) s', \]
	\[ \sigma(vs)^{\circ} = \sigma(v) s'^{\circ} = \sigma(w)^{\circ}. \]
	By Proposition \ref{prop:cubical.isometries.are.well.defined.on.vertices}, we conclude that $vs^{\circ} = w^{\circ}$ and thus $v^{\circ}$ and $w^{\circ}$ are connected by an edge. This implies that $\sigma$ induces a cubical isometry $g : X \rightarrow Y$ that sends $o$ to $o'$.
\end{proof}

Theorem \ref{thm:characterization.of.cubical.portraits} is the combined statements of Lemma \ref{lem:first.direction.of.main.theorem} and \ref{lem:converse.of.main.theorem}\\

One of the main uses of Theorem \ref{thm:characterization.of.cubical.portraits} is the construction of specific cubical isometries. When doing so, it is often convenient to construct the maps $\pre{v}{\sigma}$ for all {\bf reduced} words $v$, show that this collection of maps satisfies the five properties of cubical portraits, and use property $end$ to extend $(\pre{v}{\sigma})_{v \in \mathcal{L}_o^{red}}$ to a portrait $(\pre{v}{\sigma})_{v \in \mathcal{L}_o^{edge}}$. Since every vertex in $X$ is the endpoint of a reduced edge-path, property $end$ makes this extension unique. However, it is a-priori unclear whether this extension is still cubical. While properties $comm$, $par$, $inv$, and $end$ are very straight-forward to extend as mentioned in Remark \ref{rem:basics.on.cubical.conditions}, property $tree$ requires a bit more work. This is the content of the next Lemma.

\begin{lemma} \label{lem:extension.satisfies.tree}
	Let $(\pre{v}{\sigma})_{v \in \mathcal{L}_o^{red}}$ be a cubical portrait on $T_o^{red}$ to $\epsilon$ in $T'^{edge}_{o'}$. The unique extension to a portrait on $T_o^{edge}$ to $\epsilon$ in $T'^{edge}_{o'}$ that still satisfies $end$ also satisfies $tree$, $comm$, $par$, and $inv$.
\end{lemma}

\begin{proof}
	Let $(\pre{v}{\sigma})_{v \in \mathcal{L}_o^{edge}}$ be the extension. As discussed in Remark \ref{rem:basics.on.cubical.conditions}, $comm$, $par$, and $inv$ trivially generalise to the extension. We are left to show $tree$.

	\subsubsection*{Step 1: Show that for any $v \in \mathcal{L}_o^{edge}$, the word $\sigma(v)$ is contained in $\mathcal{L}'^{edge}_{o'}$.}
	Let $v \in \mathcal{L}_o^{edge}$. We use induction over the length of $v$. If $v = \epsilon$, then $\sigma(v) = \epsilon \in \mathcal{L}'^{edge}_{o'}$. Now suppose $\sigma(w) \in \mathcal{L}'^{edge}_{o'}$ for every $w$ of length $\leq n$ and let $v$ be of length $n+1$. We find a word $w \in \mathcal{L}^{edge}_o$ of length $n$ and $s \in \Sigma^{edge}_w$ such that $v = ws$. Let $w_{red}$ be a reduced word such that $w_{red}^{\circ} = w^{\circ}$ and thus $v^{\circ} = w_{red} s^{\circ}$. Clearly, the length of $w_{red}$ is at most the length of $w$ and thus $\leq n$. Therefore, $\sigma(w), \sigma(w_{red}) \in \mathcal{L}'^{edge}_{o'}$. Furthermore, since $w^{\circ} = w_{red}^{\circ}$, they are related by a finite sequence of the moves in Lemma \ref{lem:moves.for.edge.paths}. Since $(\pre{v}{\sigma})_{v \in \mathcal{L}_o^{edge}}$ satisfies $comm$, $par$, $inv$, and $edge$, we can apply Proposition \ref{prop:cubical.isometries.are.well.defined.on.vertices} and conclude that $\sigma(w)$ and $\sigma(w_{red})$ are related by a finite sequence of the same two moves. Since $\sigma(w), \sigma(w_{red}) \in \mathcal{L}'^{edge}_{o'}$ by induction-assumption, this implies that $\sigma(w)^{\circ} = \sigma(w_{red})^{\circ}$.
	
	By assumption, property $tree$ is satisfied on the subtree $T^{red}_o$ and thus $\pre{w_{red}}{\sigma}(s) \in \Sigma'^{edge}_{\sigma(w_{red})}$. Since $\sigma(w_{red})^{\circ} = \sigma(w)^{\circ}$, we have that $\Sigma'^{edge}_{\sigma(w_{red})} = \Sigma'^{edge}_{\sigma(w)}$ and we conclude that
	\[ \pre{w}{\sigma}(s) = \pre{w_{red}}{\sigma}(s) \in \Sigma'^{edge}_{\sigma(w_{red})} = \Sigma'^{edge}_{\sigma(w)}. \]
	We conclude that $\sigma(v) = \sigma(w) \pre{w}{\sigma}(s) \in \mathcal{L}'^{edge}_{o'}$. By induction, this finishes Step 1.
		
	\subsubsection*{Step 2: Show property $tree$.}
	Let $v \in \mathcal{L}^{edge}_o$ and $v_{red} \in \mathcal{L}^{red}_o$ such that $v^{\circ} = v_{red}^{\circ}$. By property $end$, we know that $\pre{v}{\sigma} = \pre{v_{red}}{\sigma}$. Since $v^{\circ} = v^{\circ}_{red}$ and $\sigma(v)^{\circ} = \sigma(v_{red})^{\circ}$, we know that $\Sigma_{v}^{edge} = \Sigma_{v_{red}}^{edge}$ and $\Sigma'^{edge}_{\sigma(v)} = \Sigma'^{edge}_{\sigma(v_{red})}$. Therefore,
	\[ \pre{v}{\sigma}( \Sigma_v^{edge} ) = \pre{v_{red}}{\sigma}( \Sigma_{v_{red}}^{edge} ) = \Sigma'^{edge}_{\sigma(v_{red})} = \Sigma'^{edge}_{\sigma(v)}. \]
	Therefore, $(\pre{v}{\sigma})_{v \in \mathcal{L}_o^{edge}}$ satisfies property $tree$.
\end{proof}

Lemma \ref{lem:extension.satisfies.tree} and Theorem \ref{thm:characterization.of.cubical.portraits} together imply Theorem \ref{thm:canonical.portrait.extension}. Therefore, cubical isometries $X \rightarrow Y$ can be constructed by producing cubical portraits on the tree of reduced words $T^{red}_o$.

\bibliography{mybib}
\bibliographystyle{alpha}

\end{document}